\documentclass[notitlepage,a4,10.5pt]{article}

\usepackage{graphicx}

\usepackage{amsmath}%
\usepackage{amsfonts}%
\usepackage{amssymb}%
\usepackage{mathrsfs}
\usepackage{amsthm}
\usepackage{float}

\usepackage{authblk}

\usepackage[left=2.5cm,right=2.5cm,top=3cm,bottom=3cm]{geometry} 

\usepackage{xcolor}

\newcommand{\mc}{\mathcal}
\newcommand{\cp}{\times}

\newcommand{\Norm}[1]{\left\lvert\left\lvert #1 \right\rvert\right\rvert}

\newcommand{\bol}{\boldsymbol}

\newcommand{\abs}[1]{\left\lvert{#1}\right\rvert}

\newcommand{\lr}[1]{\left({#1}\right)}

\newcommand{\mf}{\mathfrak}
\newcommand{\p}{\partial}

\newcommand{\ti}[1]{\textit{#1}}
\newcommand{\tb}[1]{\textbf{#1}}

\newtheorem{remark}{\textit{Remark}}
\newtheorem{theorem}{\textit{Theorem}}
\newtheorem{proposition}{\textit{Proposition}}
\newtheorem{lemma}{\textit{Lemma}}

\begin{document}

\title{Nested invariant tori foliating a vector field and its curl:\\ toward MHD equilibria and steady Euler flows \\in toroidal domains without continuous Euclidean isometries}
\author[1]{Naoki Sato} 
\author[2]{Michio Yamada}
\affil[1]{Graduate School of Frontier Sciences, \protect\\ The University of Tokyo, Kashiwa, Chiba 277-8561, Japan \protect\\ Email: sato\_naoki@edu.k.u-tokyo.ac.jp}
\affil[2]{Research Institute for Mathematical Sciences, \protect\\ Kyoto University, Kyoto 606-8502, Japan
\protect \\ Email: yamada@kurims.kyoto-u.ac.jp}
\date{\today}
\setcounter{Maxaffil}{0}
\renewcommand\Affilfont{\itshape\small}

    \maketitle
    \begin{abstract}
    This paper studies the problem of finding a 
    three-dimensional solenoidal vector field such that both the vector field and its curl are tangential to a given family of toroidal surfaces. We show that this  question can be translated into the problem of determining a periodic 
    solution with periodic derivatives of a two-dimensional linear elliptic second-order partial differential equation on each toroidal surface, and prove the existence of smooth solutions. Examples of smooth solutions foliated by toroidal surfaces that are not invariant under continuous Euclidean isometries are also constructed explicitly, and they are identified as  equilibria of anisotropic magnetohydrodynamics. 
    The problem examined here represents a weaker version of a fundamental mathematical problem  
    that arises in the context of magnetohydrodynamics and fluid mechanics concerning the existence of  regular equilibrium magnetic fields and steady Euler flows in bounded domains without continuous Euclidean isometries. The existence of such configurations  represents a key theoretical issue for the design of the confining magnetic field in nuclear fusion reactors known as stellarators.  
    \end{abstract}

\section{Introduction}





This paper is concerned with the equation
\begin{equation}
\left[\lr{\nabla\cp\bol{w}}\cp\bol{w}\right]\cp\nabla\Psi=\bol{0},~~~~\nabla\cdot\bol{w}=0~~~~{\rm in}~~\Omega.\label{eq1}
\end{equation}
Here, the unknown $\bol{w}\lr{\bol{x}}$ is a three-dimensional vector field with Cartesian components $w_i$, $i=1,2,3$, defined in a smooth  toroidal domain $\Omega\subset\mathbb{R}^3$ 
foliated by nested toroidal surfaces corresponding to level sets of a smooth function $\Psi\lr{\bol{x}}$
such that the bounding surface is given by $\p\Omega=\left\{\bol{x}\in\mathbb{R}^3:\Psi=\Psi_0\in\mathbb{R}\right\}$.

Equation \eqref{eq1} has the following physical meaning: 
given a set of nested toroidal surfaces $\Psi$ in the domain $\Omega$, 
can one always
find a magnetic field $\bol{B}=\bol{w}$ and an electric current $\bol{J}=\nabla\cp\bol{w}$ 
that are tangent to the level sets of $\Psi$? 
In a more general interpretation of \eqref{eq1}, both the function $\Psi$ and the shape of the
toroidal volume $\Omega$ can be treated as an unknown variables as well.

In the context of plasma physics, the problem posed by equation \eqref{eq1} represents a 
generalization of a more difficult equation, namely the magnetohydrodynamic equilibrium equation 
\begin{equation}
\lr{\nabla\cp\bol{w}}\cp\bol{w}=\nabla\Psi,~~~~\nabla\cdot\bol{w}=0~~~~{\rm in}~~{\Omega},\label{eq2} 
\end{equation}
where $P=\Psi$ plays the role of the pressure field.  
This equation describes an equilibrium state where the Lorentz force is exactly balanced by the pressure force, and its solution is crucial for the design of confining magnetic fields in nuclear fusion reactors \cite{Kruskal}. 
Notice that any solution of \eqref{eq2} is also a solution of \eqref{eq1}. 
Equation \eqref{eq2} also occurs in fluid mechanics, where it describes a steady Euler flow with velocity
field $\bol{v}=\bol{w}$ and mechanical pressure $P=-\Psi-\frac{1}{2}\bol{v}^2$ \cite{Moffatt85}.

The challenge posed by equation \eqref{eq2} 
is exemplified by the unavailability of a general theory 
concerning the existence of solutions  
\cite{LoSurdo}, although steady progress has been made since the original inception of the problem.  
As a consequence, it is not known whether regular steady fluid flows or equilibrium magnetic fields \eqref{eq2} exist in a bounded domain $\Omega$ of arbitrary shape. 
The intrinsic mathematical difficulty   
behind equation \eqref{eq2} can be understood 
in terms of characteristic surfaces. 
Indeed, if considered as a system of nonlinear first order 
partial differential equations for the unknowns $\bol{w},\Psi$, the characteristic surfaces $S$ of 
equation \eqref{eq2} are determined by the characteristic equation \cite{Yos90}
\begin{equation}
\lr{\nabla S}^2\lr{\bol{w}\cdot\nabla S}^2=0.
\end{equation}
Hence, equation \eqref{eq2} exhibits a mixed behavior, being twice elliptic
and twice hyperbolic, with the nontrivial characteristic surfaces $\lr{\bol{w}\cdot\nabla S}^2=0$ associated with hyperbolicity depending on the unknwon $\bol{w}$. These features make \eqref{eq2} 
one of the hardest partial differential equations in mathematical physics. 

Weak solutions of \eqref{eq2} have been proposed where the function $\Psi$ has a stepped profile \cite{Bruno,EncisoMHD}. In this approach, 
the toroidal domain $\Omega$ is partitioned in a set of  nested 
toroidal subdomains $\Omega_i\subset\Omega$, $i=1,...,N$,    
where the function $\Psi=\psi_i\in\mathbb{R}$ is constant, while the total pressure $\Psi+\bol{w}^2/2$ 
is continuous across the boundary separating adjacent subdomains. This construction has the advantage
that in each subdomain $\Omega_i$ equation \eqref{eq2} reduces
to the eigenvalue problem for the curl operator, 
a system of linear first order partial differential equations for which strong solutions are available  \cite{YosGiga}. However, this comes at the price of 
reduced regularity of solutions, which fall in the class $L^2\lr{\Omega}$. 
In a slightly different setting where $\Psi$ is not required to be constant on $\p\Omega$, nontrivial strong solutions of \eqref{eq2} in the class $H^1\lr{\Omega}$ have been reported in  \cite{ConstantinPasqualotto22}.  
These solutions are obtained as steady states of a Voigt approximation scheme of the time-dependent  viscous non-resistive incompressible magnetohydrodynamics equations in the limit $t\rightarrow\infty$.  

It is well known that equation \eqref{eq2} is greatly simplified whenever the vector field $\bol{w}$ and the function $\Psi$ are invariant under 
a continuous Euclidean isometry, i.e. a continuous transformation of three-dimensional space that preserves the Euclidean distance $ds^2=dx^2+dy^2+dz^2$ between points.  
Such transformations 
are characterized by a vector field $\bol{\eta}=\bol{a}+\bol{b}\cp\bol{x}$, with $\bol{a},\bol{b}\in\mathbb{R}^3$,  
representing the general solution of the equation $\mf{L}_{\bol{\eta}}ds^2=0$, where $\mf{L}$ denotes the Lie-derivative, and physically correspond to combinations of translations (generated by $\bol{a}$) and rotations (generated by $\bol{b}$).    
In the context of plasma physics, invariance under a continuous Euclidean isometry is usually referred to as a symmetry of the system. 
In formulae, the vector field $\bol{w}$ and the function $\Psi$ are invariant under a continuous Euclidean isometry whenever constant vectors $\bol{a},\bol{b}\in\mathbb{R}^3$ with $\bol{a}^2+\bol{b}^2\neq\bol{0}$ exist such that
\begin{equation}
\mf{L}_{\bol{a}+\bol{b}\cp\bol{x}}\bol{w}=\bol{0},~~~~\mf{L}_{\bol{a}+\bol{b}\cp\bol{x}}\Psi=0.\label{sym}
\end{equation}
When condition \eqref{sym} holds, equation 
\eqref{eq2} can be reduced to the Grad-Shafranov equation \cite{Eden1,Eden2}, a nonlinear second order elliptic partial differential equation for the unknown $\Psi$,  
which is assumed to satisfy Dirichlet boundary conditions on $\p\Omega$. 
Regular solutions of the Grad-Shafranov equation 
can be obtained in accordance with the theory of second order elliptic partial differential equations, thus providing regular (symmetric) solutions of \eqref{eq2}. 
Notice that in this setting the symmetry of $\Psi$ implies
the symmetry of the bounding surface $\p\Omega$, 
which corresponds to a level set of $\Psi$. 

Unfortunately, the presence of a continuous Euclidean isometry \eqref{sym} is a 
special requirement that does not apply 
in several situations of practical interest.
In particular, the confining magnetic field in nuclear fusion reactors known as stellarators 
sacrifices axial symmetry in favor of a pronounced
field line twist that aims at minimizing
plasma losses at the vessel boundary $\p\Omega$ caused by cross-field dynamics of charged particles \cite{Hel}. 
In this context, it is therefore 
necessary to understand the existence of 
solutions of \eqref{eq2} that are not endowed with continuous Euclidean symmetries. 
For completeness, it should be emphasized that even if such `asymmetric' solutions exist, they would not necessarily work as confining magnetic fields, because other requirements, such as quasisymmetry \cite{Rod}
and a small electric current, 
must be enforced on $\bol{w}$. 

The nontrivial geometrical constraints on $\bol{w}$ and $\Psi$ required for the existence of solutions of equation \eqref{eq2} in toroidal domains without continuous Euclidean symmetries have also raised 
the possibility that regular solutions of this kind  may not exist: according to the Grad conjecture \cite{Grad}, only `configurations of great geometrical symmetry' would produce well behaved equilibria. In modern plasma physics, this idea 
is usually understood as equation \eqref{sym} being a 
necessary condition for the existence of regular solutions of \eqref{eq2}.
Although Grad's conjecture remains unsolved, 
Arnold's structure theorem \cite{Arnold} 
provides a topological characterization of the field lines of any  analytic solution of \eqref{eq2} such that $\bol{w}$ and $\nabla\cp\bol{w}$ are not everywhere collinear. 
In particular, when equation \eqref{eq2} is considered in a connected analytic bounded domain $\Omega$  together with tangential boundary conditions $\bol{w}\cdot\bol{n}=0$ on $\p\Omega$,
where $\bol{n}$ is the unit outward normal to $\p\Omega$, any contour of $\Psi$ that does not intersect the boundary $\p\Omega$ and such that $\nabla\Psi\neq \bol{0}$ is a two-dimensional torus. This result is also the reason why 
toroidal volumes such that $\Psi$ is constant on the boundary are considered in the study of equation \eqref{eq2}.  
The fact that a simpler topology where level sets of $\Psi$ are spherical is not consistent with \eqref{eq2} can also be understood through the hairy ball theorem \cite{Eisen}, 
which precludes the existence of a continuous non-vanishing vector field always tangent to a $2$-sphere. 

Considering the challenge posed by equation \eqref{eq2} described above, here we examine the  simplified problem of equation \eqref{eq1}. 
Observe that while in \eqref{eq2} the
magnitude of the component of $\lr{\nabla\cp\bol{w}}\cp\bol{w}$ along $\nabla\Psi$ is exactly $\abs{\nabla\Psi}$, no such requirement appears in \eqref{eq1}. As it will be shown later, under suitable assumptions this simplifies the mathematical difficulty by a `half', since the governing equations  
can be reduced from two to one. 
Notice that studying equation \eqref{eq1} may provide useful information concerning the nature of the space of solutions of equation \eqref{eq2}. 
Indeed, any conditions preventing the existence of solutions of \eqref{eq1} would also apply to \eqref{eq2}.
Furthermore, if regular solutions of \eqref{eq1} could be obtained, it would be possible to identify the geometrical obstruction preventing such solutions from solving \eqref{eq2} as well. 

The strategy we adopt to examine equation \eqref{eq1} is to reduce the equation by a 
Clebsch representation \cite{YosClebsch,YosMor} of the vector field $\bol{w}$ through a pair of Clebsch potentials $\lr{\Psi,\Theta}$ that reflect the foliated ($\nabla\Psi\cdot\bol{w}=0$) and solenoidal ($\nabla\cdot\bol{w}=0$) nature  of the candidate solution $\bol{w}=\nabla\Psi\cp\nabla\Theta$.  
This approach has the advantage that 
the topology of the foliation  
associated with the (given) function $\Psi$ can be enforced a priori, leaving the analysis of the existence of solutions as an independent issue for the Clebsch potential $\Theta$. 
In particular, we prove the following: 

\begin{theorem}
Let $\Omega\subset\mathbb{R}^3$ denote a bounded domain.  
Assume that the bounding surface $\p\Omega$ is a hollow torus corresponding to 
two distinct level sets of a smooth function $\Psi\in C^{\infty}
({\Omega})$, with $\nabla\Psi\neq\bol{0}$ in $\Omega$, 
and that 
level sets of $\Psi$ foliate $\Omega$ with nested toroidal surfaces endowed with angle coordinates $\mu,\nu$ with smooth gradients  $\nabla\mu,\nabla\nu\in C^{\infty}\lr{\Omega}$. 
Then, the system of partial differential equations
\begin{equation}
\left[\lr{\nabla\cp\bol{w}}\cp\bol{w}\right]\cp\nabla\Psi=\bol{0},
~~~~\nabla\cdot\bol{w}=0~~~~{\rm in}~~\Omega,\label{eq0}
\end{equation}
admits a nontrivial solution
$\bol{w}\in C^{\infty}\lr{\Omega}$ such that  $\bol{w}$ and $\nabla\cp\bol{w}$ are not everywhere collinear.  
\end{theorem}

The proof of theorem 1 follows by observing that,
upon introducing the Clebsch representation
$\bol{w}=\nabla\Psi\cp\nabla\Theta$ of the solution, equation \eqref{eq0} reduces to a single linear 
elliptic second-order partial differential equation on each toroidal surface $\Psi={\rm constant}$ for the unknown $\Theta$ in a periodic domain.
Regular periodic solutions of these equations can be obtained 
by elliptic theory. 
A global solution $\Theta$ can then be constructed by smoothly joining solutions corresponding to different toroidal surfaces, thus providing a smooth solution $\bol{w}$ of \eqref{eq0} in a hollow toroidal volume $\Omega$.  

The present paper is organized as follows. 
In section 2, basic aspects of equation \eqref{eq1} are discussed, the reduction of the equation through Clebsch potentials to a three-dimensional linear second-order degenerate elliptic partial differential equation is presented, and a corresponding variational formulation is derived.
In section 3, the reduced equation is reformulated as a family of two-dimensional linear elliptic second-order partial differential equations, each correponding to a given toroidal surface. 
Theorem 1 is proven in section 4, while an  example of numerical solution is obtained in section 5.
Smooth solutions of equation \eqref{eq1} with nested toroidal surfaces and without continuous Euclidean isometries are then constructed explicitly in section 6. It is further shown that such solutions can be ragarded as a equilibrium magnetic fields within the framework of anisotropic magnetohydrodynamics. 
Section 7 presents some considerations on the
application of the theory to the study of equation \eqref{eq2}. Concluding remarks are given in section 8.


\section{General properties of the equation}
The aim of this section is to discuss some basic properties of equation \eqref{eq1},
as well as a variational formulation of the problem in terms of Clebsch potentials.  
For the purpose of this section, we shall assume that all involved quantities can be differentiated as many times as necessary. 

First, it should be noted that the main difficulty in system \eqref{eq1} is not the
requirement $\left[\lr{\nabla\cp\bol{w}}\times\bol{w}\right]\cp\nabla\Psi=\bol{0}$ that both $\bol{w}$ and $\nabla\cp\bol{w}$ lie on the same 
surface $\Psi={\rm constant}$ per se, but rather its combination with 
the solenoidal condition $\nabla\cdot\bol{w}=0$. 
Indeed, if one drops the equation $\nabla\cdot\bol{w}=0$, any 
vector field of the form
\begin{equation}
\bol{w}=f\lr{\Psi,\alpha}\nabla\alpha +g\lr{\Psi,\beta}\nabla\beta,
\end{equation}
where $\alpha$ and $\beta$ are functions with the property that $\nabla\Psi\cdot\nabla\alpha=\nabla\Psi\cdot\nabla\beta=0$, 
is a nontrivial solution of $\left[\lr{\nabla\cp\bol{w}}\times\bol{w}\right]\cp\nabla\Psi=\bol{0}$ where the proportionality coefficient $\lambda$ 
between $\lr{\nabla\cp\bol{w}}\cp\bol{w}$ and $\nabla\Psi$ is given by 
\begin{equation}
\lambda=-\frac{1}{2}\frac{\p f^2}{\p\Psi}\abs{\nabla\alpha}^2
-\frac{1}{2}\frac{\p g^2}{\p\Psi}\abs{\nabla\beta}^2
-\frac{\p \lr{fg}}{\p\Psi}\nabla\alpha\cdot\nabla\beta.
\end{equation}
Such configurations can be constructed explicitly. 
For example, introduce cylindrical coordinates $\lr{r,\varphi,z}$ and consider a family of axially symmetric toroidal surfaces with circular cross section and major radius $r_0>0$ corresponding to level sets of the function
\begin{equation}
\Psi_0=\frac{1}{2}\left[\lr{r-r_0}^2+z^2\right].\label{Psi0}
\end{equation}
The axial symmetry of $\Psi_0$ can be broken by introducing a small displacement, 
\begin{equation}
\Psi_{\epsilon}=\Psi_0+\frac{1}{2}\epsilon\sin\lr{m\varphi},~~~~m\in\mathbb{Z},~~~~m\neq0,\label{PsiEx}
\end{equation}
where $\epsilon>0$ is a small real constant. 
In particular, observe that for a sufficiently small $\epsilon$ 
level sets of \eqref{PsiEx} generate toroidal surfaces. 
Furthermore, due to the dependence on the toroidal angle $\varphi$, these surfaces are not invariant under continuous Euclidean isometries, i.e.  some appropriate combination of translations and rotations. 
Indeed, recalling that $\bol{a}+\bol{b}\times\bol{x}$ with $\bol{a},\bol{b}\in\mathbb{R}^3$ is the generator of continuous Euclidean isometries, one sees that the only solution of the equation
\begin{equation}
\mf{L}_{{\bol{a}+\bol{b}\cp\bol{x}}}\Psi_{\epsilon}=\lr{\bol{a}+\bol{b}\cp\bol{x}}\cdot\nabla\Psi_{\epsilon}=0,
\end{equation}
is $\bol{a}=\bol{b}=\bol{0}$.
Let us verify this fact explicitly. We have
\begin{equation}
\begin{split}
\mf{L}_{{\bol{a}+\bol{b}\cp\bol{x}}}\Psi_{\epsilon}=&
\lr{a_x+b_yz-b_zy}\left[
\lr{1-\frac{r_0}{r}}x-\frac{1}{2}\epsilon m\cos\lr{m\varphi}\frac{y}{r^2}\right]
\\&+\lr{a_y+b_zx-b_x z}\left[\lr{1-\frac{r_0}{r}}y+\frac{1}{2}\epsilon m \cos\lr{m\varphi}\frac{x}{r^2}\right]+z\lr{a_z+b_xy-b_yx},\label{Lie11}
\end{split}
\end{equation}
where $a_x$, $a_y$, $a_z$, $b_x$, $b_y$, and $b_z$ are the Cartesian components of $\bol{a}$ and $\bol{b}$. 
Next, evaluating the expression above along the positive $x$-axis where $\varphi=y=z=0$ and $x=r$ gives
\begin{equation}
\mf{L}_{{\bol{a}+\bol{b}\cp\bol{x}}}\Psi_{\epsilon}=
\frac{\epsilon m}{2}b_z-a_xr_0+a_x x+a_y\frac{\epsilon m}{2x}. 
\end{equation}
Since $x$ is not constant, this quantity vanishes only if $a_x=a_y=b_z=0$. 
Now, the surviving terms in equation \eqref{Lie11} 
are those involving $a_z$, $b_x$, and $b_y$. 
On the toroidal section $\varphi=0$, equation \eqref{Lie11} 
therefore becomes
\begin{equation}
\mf{L}_{{\bol{a}+\bol{b}\cp\bol{x}}}\Psi_{\epsilon}=-b_yr_0z
-b_xz\frac{\epsilon m}{2r}+za_z.
\end{equation}
It follows that the expression above vanishes if $b_x=0$ 
and $a_z=b_yr_0$. 
Finally, consider the toroidal section $\varphi=\pi/2$. Here, 
\begin{equation}
\mf{L}_{{\bol{a}+\bol{b}\cp\bol{x}}}\Psi_{\epsilon}=-b_yz\frac{\epsilon m}{2r}\cos\lr{\frac{m\pi}{2}}+z b_yr_0.
\end{equation}
This quantity vanishes for arbitrary $r$ and $z$ only if $b_y=a_z/r_0=0$. 
Hence, $\bol{a}=\bol{b}=\bol{0}$, which implies that level sets of \eqref{PsiEx} are not invariant under continuous Euclidean isometries. 

Now define the toroidal domain $\Omega$ as the volume enclosed by a contour of the function $\Psi_{\epsilon}$ in equation \eqref{PsiEx}, and consider the vector field 
\begin{equation}
\bol{w}=f\lr{\Psi_{\epsilon}}\nabla\alpha,\label{wEx}
\end{equation}
where $f$ is some function of $\Psi_{\epsilon}$ and $\alpha=\arctan\left[z/\lr{r-r_0}\right]$. 
Since $\nabla\Psi_{\epsilon}\cdot\nabla\alpha=0$, it readily follows that
\begin{equation}
\lr{\nabla\cp\bol{w}}\cp\bol{w}=-\frac{1}{2}\frac{\p f^2}{\p\Psi_{\epsilon}}\abs{\nabla\alpha}^2\nabla\Psi_{\epsilon}=-\frac{1}{2\left[2\Psi_{\epsilon}-\epsilon\sin\lr{m\varphi}\right]}\frac{\p f^2}{\p\Psi_{\epsilon}}\nabla \Psi_{\epsilon}, 
\end{equation}
although 
\begin{equation}
\nabla\cdot\bol{w}=f\lr{\Psi_{\epsilon}}\Delta\alpha=-\frac{zf\lr{\Psi_{\epsilon}}}{r\left[2\Psi_{\epsilon}-\epsilon \sin\lr{m\varphi}\right]},
\end{equation}
which does not vanish in general. 
A plot of the vector field \eqref{wEx} and its modulus on a level set of $\Psi_{\epsilon}$ is given in figure \ref{fig1}. 
\begin{figure}[h]
\hspace*{-0cm}\centering
\includegraphics[scale=0.5]{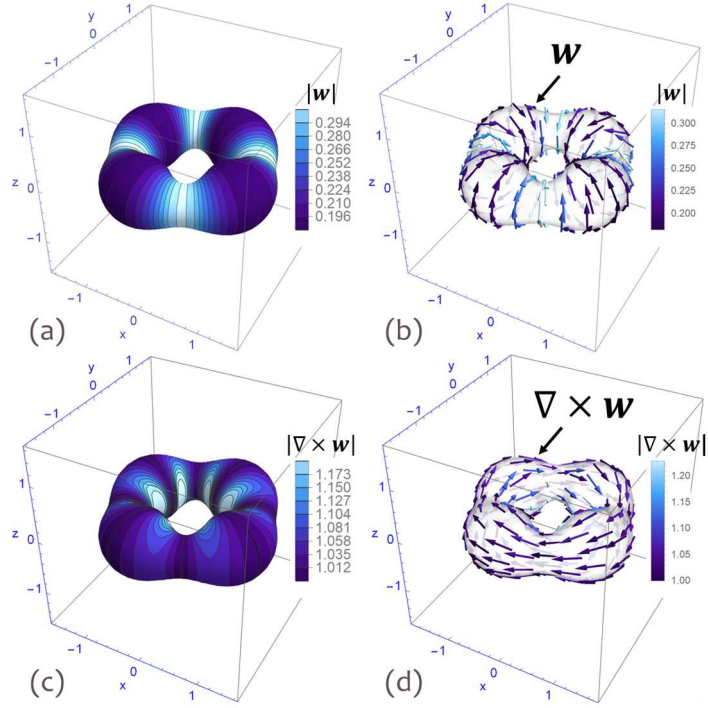}
\caption{\footnotesize (a) Contour plot of $\abs{\bol{w}}$ over the level set $\Psi_{\epsilon}=0.1$. 
(b) Vector plot of $\bol{w}$ over the level set $\Psi_{\epsilon}=0.1$. 
(c) Contour plot of $\abs{\nabla\cp\bol{w}}$ over the level set $\Psi_{\epsilon}=0.1$. 
(d) Vector plot of $\nabla\cp\bol{w}$ over the level set $\Psi_{\epsilon}=0.1$.
In (a), (b), (c), and (d) 
$\bol{w}$ is defined by equation \eqref{wEx}, $f=\Psi_{\epsilon}$, and 
the values $r_0=1$, $\epsilon=0.1$, and $m=4$ are used.}
\label{fig1}
\end{figure}
This example shows that if the space of solutions is not restricted to solenoidal vector fields, 
smooth vector fields obeying $\left[\lr{\nabla\cp\bol{w}}\cp\bol{w}\right]\cp{\nabla\Psi}=\bol{0}$ in some bounded region can be obtained 
in a rather straightforward fashion. 

We now return to the original problem \eqref{eq1}.
Observe that whenever 
$\lr{\nabla\cp\bol{w}}\cp\bol{w}\neq\bol{0}$ any solution of \eqref{eq1} must satisfy $\bol{w}\cdot\nabla\Psi=0$, or
\begin{equation}
\bol{w}=\nabla\Psi\cp\bol{q},\label{wq}
\end{equation}
for some vector field $\bol{q}\lr{\bol{x}}$. 
Since $\nabla\cdot\bol{w}=-\nabla\Psi\cdot\nabla\cp\bol{q}$, 
a straightforward way to ensure that
$\nabla\cdot\bol{w}=0$ is to demand that $\nabla\cp\bol{q}=\bol{0}$. 
For a given 
solenoidal vector field $\bol{w}$ 
in
a small enough neighborhood 
the vector field $\bol{q}$ can always be found 
in the form $\bol{q}=\nabla \theta$, with $\theta$ a single-valued function, provided that $\bol{w}$ 
is sufficiently regular (Darboux theorem \cite{Dhb1,Dhb2}). 
This implies that locally solutions of \eqref{eq1} have the form $\bol{w}=\nabla\Psi\cp\nabla\theta$. 
We shall therefore consider
candidate global solutions of the type 
\begin{equation}
\bol{w}=\nabla\Psi\cp\nabla\Theta,\label{wCl}
\end{equation}
where $\Theta$ is allowed to be a multivalued (angle) variable. 
More precisely, we write $\nabla\Theta$ to denote an element of the kernel of the curl operator, $\nabla\Theta\in {\rm Ker\lr{curl}}$. 
In the following, we shall refer to scalar quantities such as $\Psi$ and $\Theta$ used to express a vector field as Clebsch potentials, and to the form \eqref{wCl} as a Clebsch representation of the solution $\bol{w}$ (see \cite{YosClebsch} for additional details on Clebsh representations and their completeness).
It is now clear that finding a solution of \eqref{eq1} in the Clebsch form \eqref{wCl} is tantamount to
determining two vector fields $\nabla\Theta$ and $\bol{p}$ such that
\begin{equation}
\nabla\cp\lr{\nabla\Psi\cp\nabla\Theta}=\nabla\Psi\cp\bol{p}~~~~{\rm in}~~\Omega.\label{wp}
\end{equation}
Equation \eqref{wp} is equivalent to demanding that $\nabla\cp\bol{w}$ does not have any component in the direction of $\nabla\Psi$, i.e.
\begin{equation}
\nabla\Psi\cdot\nabla\cp\lr{\nabla\Psi\cp\nabla\Theta}=0~~~~{\rm in}~~\Omega.
\end{equation}
By standard vector identities, we thus arrive at 
\begin{equation}
\nabla\cdot\left[\nabla\Psi\cp\lr{\nabla\Theta\cp\nabla\Psi}\right]=0~~~~{\rm in}~~\Omega.\label{eq1X}
\end{equation}
Hence, solutions of \eqref{eq1} 
with Clebsch representation 
\eqref{wCl} are solutions of \eqref{eq1X}. 
In the reminder of this paper 
we shall therefore concentrate our efforts on the study of equation \eqref{eq1X} under appropriate boundary conditions for the variable $\Theta$. 

For a given $\Psi$, equation \eqref{eq1X} is a second order degenerate elliptic
partial differential equation for the unknown $\Theta$. 
The equation is also linear, in contrast with the nonlinearity of system \eqref{eq1}. 
To see this, observe that \eqref{eq1X} can be written as
\begin{equation}
\begin{split}
&\abs{\nabla\Psi}^2\Delta\Theta-\nabla\Psi\cdot\lr{\nabla\Psi\cdot\nabla}\nabla\Theta-
\nabla\Theta\cdot\lr{\nabla\Psi\cdot\nabla}\nabla\Psi+\nabla\abs{\nabla\Psi}^2\cdot\nabla\Theta-\lr{\nabla\Psi\cdot\nabla\Theta}\Delta\Psi\\&=\sum_{i,j=1}^3
\abs{\nabla\Psi}^2\lr{\delta_{ij}-\frac{\Psi_i\Psi_j}{\abs{\nabla\Psi}^2}}\Theta_{ij}+\sum_{i=1}^3\lr{\frac{1}{2}\abs{\nabla\Psi}^2_i
-\Psi_i\Delta\Psi}\Theta_i=0~~~~{\rm in}~~\Omega,
\end{split}
\end{equation}
where lower indices have been used as a shorthand notation for partial derivatives with respect to Cartesian coordinates $\lr{x,y,z}=\lr{x^1,x^2,x^3}$,  
e.g. $\Theta_1=\p\Theta/\p x^1=\p\Theta/\p x$. 
The coefficient matrix
\begin{equation}
\mf{a}_{ij}=\abs{\nabla\Psi}^2\lr{\delta_{ij}-\frac{\Psi_i\Psi_j}{\abs{\nabla\Psi}^2}},~~~~i,j=1,2,3,
\end{equation}
is symmetric and positive semi-definite since
\begin{equation}
\mf{a}_{ij}\xi^i\xi^j=\abs{\nabla\Psi\cp\bol{\xi}}^2\geq 0,~~~~\bol{\xi}\in\mathbb{R}^3,~~\bol{x}\in\Omega,
\end{equation}
and thus defines a degenerate elliptic 
differential operator. 
The degeneracy of the solution is evident from the fact that if $\Theta$ 
is a solution of \eqref{eq1X}, so is $\Theta+f\lr{\Psi}$, with $f$ a function of $\Psi$.
In particular, it should be emphasized that the degeneracy is not expected to prevent the existence of solutions, but simply to affect their uniqueness.

Equation \eqref{eq1X} also admits a variational formulation. 
Indeed, defining the magnetic energy (kinetic energy in the fluid analogy)
\begin{equation}
E_{\Omega}=\frac{1}{2}\int_{\Omega}\bol{w}^2\,dV=\frac{1}{2}\int_{\Omega}\abs{\nabla\Psi\cp\nabla\Theta}^2\,dV,\label{E}
\end{equation}
where $dV$ is the volume element in $\mathbb{R}^3$, 
and assuming that variations $\delta\Theta$ 
vanish on the bounding surface $\p\Omega$, one obtains
\begin{equation}
\delta E_{\Omega}=\int_{\Omega}\nabla\delta\Theta\cdot\lr{\nabla\Psi\cp\nabla\Theta}\cp\nabla\Psi\,dV=-\int_{\Omega}\delta\Theta\nabla\cdot\left[\nabla\Psi\cp\lr{\nabla\Theta\cp\nabla\Psi}\right]\,dV.
\end{equation}
Hence, stationary points of the energy $E_{\Omega}$ correspond to 
solutions of \eqref{eq1X}. 

\section{Reformulation as an elliptic equation on a toroidal surface}
As outlined in the introduction, we aim to
remove the degeneracy of equation \eqref{eq1X} by 
reducing it to a linear two-dimensional second-order elliptic partial differential equation over each toroidal surface $\Psi={\rm constant}$. 
The degeneracy 
can be effectively removed, for example, by 
fixing the mean value $\langle\Theta\rangle$ of 
the unknown $\Theta$ over the surface. 
A unique solution of \eqref{eq1X} 
can then obtained by patching 
two-dimensional solutions corresponding to different toroidal surfaces.
In order to implement this construction, we introduce curvilinear coordinates $\lr{x^1,x^2,x^3}=\lr{\mu,\nu,\Psi}$ with $\mu,\nu\in \left[0,2\pi\right)$ angle coordinates spanning the toroidal surfaces $\Psi={\rm  constant}$, $\p_i$, $i=1,2,3$, the corresponding tangent vectors,   $J=\nabla\mu\cdot\nabla\nu\cp\nabla\Psi$ the Jacobian determinant of the transformation, and $g_{ij}=\p_i\cdot\p_j$ the covariant metric tensor.  
Using these quantities,
equation \eqref{eq1X}
restricted to the surface $\Sigma_{\Psi_0}=\left\{\bol{x}\in\Omega:\Psi\lr{\bol{x}}=\Psi_0\in\mathbb{R}\right\}$
takes the form
\begin{equation}
\frac{\p}{\p\mu}\left[
J\lr{g_{\nu\nu}\frac{\p\Theta}{\p\mu}-g_{\mu\nu}\frac{\p\Theta}{\p\nu}}\right]
+\frac{\p}{\p\nu}
\left[
J\lr{g_{\mu\mu}\frac{\p\Theta}{\p\nu}-g_{\mu\nu}\frac{\p\Theta}{\p\mu}}
\right]=0~~~~{\rm in}~~\Sigma_{\Psi_0}.\label{eq1Y}
\end{equation}
Notice that in this notation $g_{11}=g_{\mu\mu}$, $g_{12}=g_{\mu\nu}$, and $g_{22}=g_{\nu\nu}$. 
Let us verify that the two-dimensional second order partial differential equation \eqref{eq1Y} is 
elliptic. 
First, observe that equation \eqref{eq1Y}
can be written as
\begin{equation}
\begin{split}
&g_{\nu\nu}\Theta_{\mu\mu}
-2g_{\mu\nu}\Theta_{\mu\nu}
+g_{\mu\mu}\Theta_{\nu\nu}
+\left[
\frac{J_{\mu}}{J}g_{\nu\nu}+\frac{\p g_{\nu\nu}}{\p\mu}-\frac{J_{\nu}}{J}g_{\mu\nu}-\frac{\p g_{\mu\nu}}{\p\nu}
\right]
\Theta_{\mu}
\\&+\left[\frac{J_\nu}{J}g_{\mu\mu}+\frac{\p g_{\mu\mu}}{\p\nu}-\frac{J_{\mu}}{J}g_{\mu\nu}-\frac{\p g_{\mu\nu}}{\p\mu}
\right]\Theta_\nu=0
~~~~{\rm in}~~\Sigma_{\Psi_0}.\label{eq1Y2}
\end{split}
\end{equation}
Equation \eqref{eq1Y2} has the form
\begin{equation}
\sum_{i,j=1}^2a_{ij}\Theta_{ij}+{\rm lower~order~ terms}=0,
\end{equation}
where the coefficient matrix $A$ with components $a_{ij}$, $i,j=1,2$, is given by 
\begin{equation}
A=\begin{bmatrix}
g_{\nu\nu}&-g_{\mu\nu}\\
-g_{\mu\nu}&g_{\mu\mu}
\end{bmatrix}\label{A}
\end{equation}
Evidently, $A=A^T$. Furthermore, the eigenvalues of $A$ are given by
\begin{equation}
\lambda_{\pm}=\frac{{\rm Tr}A\pm\sqrt{\lr{{\rm Tr}A}^2-4{\rm det}A}}{2},
\end{equation}
with 
\begin{equation}
{\rm Tr}A=g_{\mu\mu}+g_{\nu\nu}>0,~~~~{\rm det}A=g_{\nu\nu}g_{\mu\mu}-g^2_{\mu\nu}=\abs{\p_{\nu}\cp\p_{\mu}}^2>0.
\end{equation}
Both eigenvalues are real and positive with $\lambda_{+}\geq \lambda_{-}>0$ because
\begin{equation}
\lr{{\rm TrA}}^2>\lr{{\rm TrA}}^2-4{\rm det}A=\lr{g_{\mu\mu}-g_{\nu\nu}}^2+4g_{\mu\nu}^2\geq 0.
\end{equation}
It therefore follows that 
equation \eqref{eq1Y} is strictly elliptic. Indeed,
any vector $\bol{\xi}\in\mathbb{R}^2$ can be decomposed on the basis of normalized eigenvectors  $\lr{\bol{e}_{+},\bol{e}_{-}}$ so that 
\begin{equation}
a_{ij}\xi^i\xi^j\geq \lambda_{-}\abs{\bol{\xi}}^2\geq0,~~~~\bol{\xi}\in\mathbb{R}^2,~~~~\lr{\mu,\nu}\in\left[0,2\pi\right),~~~~\Psi=\Psi_0.
\end{equation}


Thanks to the strictly elliptic nature of the differential operator, once appropriate boundary conditions are enforced, 
solutions $\Theta$
of equation \eqref{eq1Y} 
exist and are unique 
(see e.g. \cite{Gil}). 
Furthermore, as it will be shown later, 
solutions corresponding 
to different toroidal surfaces can be patched together to obtain a solution in the three-diemensional volume $\Omega$. 
However, not all boundary conditions can be used to produce nontrivial vector fields $\bol{w}$ in $\Omega$. For example, 
defining the doubly periodic domain $D=\lr{0,2\pi}^2$, 
Dirichlet boundary conditions $\Theta=0$ on 
$\p D$ result in the trivial solution $\Theta=0$, and thus $\bol{w}=\bol{0}$. 
Furthermore, even if 
a set of boundary conditions
results in a nontrivial $\Theta$, 
there is no guarantee that 
the corresponding vector field 
$\nabla\Theta$ is a periodic function of the variables $\mu$ and $\nu$,  a condition that is necessary for the continuity of the solution $\bol{w}=\nabla\Psi\cp\nabla\Theta$ in $\Omega$.   
In principle, these difficulties can be avoided as follows. First, performing the change of variables
\begin{equation}
\Theta=\mu+\rho,\label{cov}
\end{equation}
in equation \eqref{eq1Y} gives 
\begin{equation}
\frac{\p}{\p\mu}\left[
J\lr{g_{\nu\nu}\frac{\p\rho}{\p\mu}-g_{\mu\nu}\frac{\p\rho}{\p\nu}}\right]
+\frac{\p}{\p\nu}
\left[
J\lr{g_{\mu\mu}\frac{\p\rho}{\p\nu}-g_{\mu\nu}\frac{\p\rho}{\p\mu}}
\right]=\frac{\p}{\p\nu}\lr{Jg_{\mu\nu}}-\frac{\p}{\p\mu}\lr{Jg_{\nu\nu}}~~~~{\rm in}~~D.\label{redeq0}
\end{equation}
Note that equation \eqref{redeq0} is strictly elliptic because it shares the same coefficient matrix $A$ with equation \eqref{eq1Y}. 
Furthermore, $\Sigma_{\Psi_0}$ has been replaced with $D$ to emphasize that
the problem is being considered within the domain of the angles $\mu$ and $\nu$. 
Next, consider a periodic solution $\rho=\sum_{m,n}c_{mn}\lr{\Psi_0}e^{{\rm i}\lr{m\mu+n\nu}}$ such that
the integral
\begin{equation}
\langle\rho\rangle=\int_{D}d\mu d\nu\rho=0, 
\end{equation}
vanishes, i.e. such that $c_{00}=0$ (this latter condition ensures that the solution $\rho$ is unique on the  toroidal surface $\Psi_0$).
If such solution $\rho$
could be found, the corresponding $\Theta$ 
would be nontrivial since
 its gradient $\nabla\Theta=\nabla\mu+\nabla\rho$ would be 
 a periodic function of both $\mu$ and $\nu$
 such that $\langle\Theta_{\mu}\rangle=\langle 1+\rho_{\mu}\rangle=4\pi^2$. Of course, other changes of variables, such as $\Theta=M\lr{\Psi}\mu+N\lr{\Psi}\nu+\rho$, could be used as well. In fact, by appropriate choice of the functions $M\lr{\Psi}$ and $N\lr{\Psi}$ one can control the 
 rotational transform (the number of poloidal transits per toroidal transit of a field line on each level set of $\Psi$) of the solution $\bol{w}$.

It is now clear that the original problem \eqref{eq1} has been reduced to the existence of a periodic solution (with periodic derivatives) of equation \eqref{redeq0} that depends in a regular fashion on the surface label $\Psi$.
Although the coefficients 
appearing in 
equation \eqref{redeq0} are periodic functions of $\mu$ and $\nu$, enforcing a boundary condition such as $\rho\lr{0,\nu,\Psi_0}=\rho\lr{2\pi,\nu,\Psi_0}=\rho\lr{\mu,0,\Psi_0}=\rho\lr{\mu,2\pi,\Psi_0}=0$ is not enough to ensure 
the periodicity of the partial derivatives $\rho_{\mu}$,   $\rho_{\nu}$, $\rho_{\Psi}$, and so on. In other words, 
the corresponding solution will not generally correspond to a converging Fourier series. 
Therefore, the regularity of the solution $\bol{w}=\nabla\Psi\cp\nabla\Theta$ will reflect the degree of periodicity of the solution $\rho$ and its partial derivatives.
In particular, denoting with $\lr{\p_{\mu},\p_{\nu},\p_{\Psi}}$ the tangent basis, observe that 
\begin{subequations}
\begin{align}
\bol{w}=&J\lr{\frac{\p\Theta}{\p\mu}\p_{\nu}-\frac{\p\Theta}{\p\nu}\p_{\mu}},\\
\nabla\cp\bol{w}=&
J\left\{\frac{\p}{\p\Psi}\left[
J\lr{g_{\mu\nu}\frac{\p\Theta}{\p\mu}-g_{\mu\mu}\frac{\p\Theta}{\p\nu}}
\right]
-\frac{\p}{\p\mu}\left[
J\lr{g_{\nu\Psi}\frac{\p\Theta}{\p\mu}-g_{\Psi\mu}\frac{\p\Theta}{\p\nu}}
\right]
\right\}\p_{\nu}
\notag\\&+J\left\{
\frac{\p}{\p\Psi}\left[
J\lr{g_{\mu\nu}\frac{\p\Theta}{\p\nu}-g_{\nu\nu}\frac{\p\Theta}{\p\mu}}
\right]
+\frac{\p}{\p\nu}\left[
J\lr{g_{\nu\Psi}\frac{\p\Theta}{\p\mu}-g_{\Psi\mu}\frac{\p\Theta}{\p\nu}}
\right]
\right\}\p_{\mu}. 
\end{align}\label{wcurlw}
\end{subequations}
Hence, for $\bol{w}$ and $\nabla\cp\bol{w}$ to be continuous in $\Omega$ it is necessary that the partial derivatives $\Theta_{\mu}=1+\rho_{\mu}$, $\Theta_{\nu}=\rho_{\nu}$, $\Theta_{\mu\mu}=\rho_{\mu\mu}$, $\Theta_{\mu\nu}=\rho_{\mu\nu}$, and $\Theta_{\nu\nu}=\rho_{\nu\nu}$
are periodic functions of $\mu$ and $\nu$. 
Conversely, if 
they fail to be periodic, $\bol{w}$ and $\nabla\cp\bol{w}$ will exhibit discontinuities on each toroidal surface in correspondence of the curves $\gamma_{\p D}=\left\{\bol{x}\in\Omega:\lr{\mu,\nu}\in\p D,\Psi=\Psi_0\right\}$. 

\section{Proof of the main theorem}

The purpose of this section is to prove theorem 1. 
As noted at the end of the previous section, 
in order to obtain a regular solution $\bol{w}$ of \eqref{eq1} in the domain $\Omega$, any solution $\rho$ of
\eqref{redeq0} must have periodic derivatives in the angles $\mu$ and $\nu$. 
This implies that the standard theory for elliptic partial differential equations
cannot be applied in a straightforward fashion because Dirichlet boundary conditions for $\rho$ do not guarantee the periodicity of its partial derivatives.  
Since there are no requirements on the boundary values that the 
function $\Theta=\mu+\rho$ should take on $\p D$, 
the idea is to 
construct   
a weak periodic solution
of equation \eqref{redeq0} in a two-dimensional lattice extending over $\mathbb{R}^2$ with unit cell $D$ by introducing an appropriate Hilbert space $H^1_{\rm per}\lr{D}\subset H^1\lr{D}$ containing  periodic functions. 
Then, interior regularity 
can be used to infer smoothness of weak solutions,
and thus periodicity of their derivatives. 

To carry out the program above, we begin by proving the following lemma:
\begin{lemma}
Let $V=D\times U$ denote a doubly periodic three-dimensional domain   
spanned by coordinates $\mu,\nu\in \left[0,2\pi\right)$, $\Psi\in {U}$, with $D=\lr{0,2\pi}^2$ and $U\subset \mathbb{R}$ a bounded open interval.
Define $\lr{x^1,x^2}=\lr{\mu,\nu}$. 
Let $\alpha^{ij}\in C^{\infty}\lr{\mathbb{R}^2\times U}$, $i,j=1,2$, and $S\in C^{\infty}\lr{\mathbb{R}^2\times U}$ be smooth  functions which are periodic in $D$. 
Further assume that 
$\langle S\rangle=\int_D S\,d\mu d\nu=0$, and that 
$\alpha^{ij}$ is strictly elliptic on each level set of $\Psi$, i.e.  
\begin{equation}
\alpha^{ij}\xi_i\xi_j\geq \lambda \abs{\bol{\xi}}^2,~~~~ \bol{\xi}\in\mathbb{R}^2,~~~~\mu,\nu\in[0,2\pi),~~~~\Psi\in {U}, 
\end{equation}
for some positive constant $\lambda$. 
Then, the boundary value problem
\begin{subequations}
\begin{align}
&\frac{\p}{\p x^i}\lr{\alpha^{ij}\frac{\p\rho}{\p x^j}}=S,~~~~\langle\rho\rangle=\int_0^{2\pi}d\mu\int_{0}^{2\pi}d\nu\rho =0
~~~~{\rm in}~~V,\label{PDE}\\ 
&\rho~~{\rm periodic}~~{\rm in}~~D
,\label{PBC}
\end{align}\label{thmX1}
\end{subequations}
 admits a unique periodic solution $\rho\in 
 C^{\infty}\lr{\mathbb{R}^2\times U}$
 with periodic derivatives of all orders. 
 In particular, for fixed $\Psi\in U$ 
 the function of two variables   
 $\rho^{\Psi}\lr{\mu,\nu}=\rho\lr{\mu,\nu,\Psi}$ satisfies $\rho^{\Psi}\in C^{\infty}\lr{\mathbb{R}^2}\cap H^1_{\rm per}\lr{D}$. Here, 
 \begin{equation}
H^1_{\rm per}\lr{D}=\left\{
\rho^{\Psi}\in H^1\lr{D}
;\langle\rho^{\Psi}\rangle=0,\rho^{\Psi}~~{\rm periodic}~~{\rm in}~~D
\right\}. 
 \end{equation}
\end{lemma}

\begin{proof}
First, observe that 
a function $\rho$ is periodic in $D$
provided that it takes the same values at opposite sides of the square, 
\begin{equation}
\rho\lr{0,\nu,\Psi}=\rho\lr{2\pi,\nu,\Psi},~~~~\rho\lr{\mu,0,\Psi}=\rho\lr{\mu,2\pi,\Psi}.
\end{equation}
Considering a two-dimensional lattice with unit cell $D$, evidently a periodic solution satisfies the property
\begin{equation}
\rho\lr{\mu,\nu,\Psi}=\rho\lr{\mu+2\pi m,\nu+2\pi n,\Psi},~~~~\forall m,n\in \mathbb{Z}.
\end{equation}
Hence, if derivatives of $\rho$ exist, they are periodic functions as well. 
Next, notice that 
for each value of $\Psi\in U$  
the strict ellipticity of $\alpha^{ij}$, the regularity and periodicity 
of both $\alpha^{ij}$ and $S$, and the condition $\langle S\rangle=0$  
guarantee that the  boundary value problem
\begin{subequations}
\begin{align}
&\frac{\p}{\p x^i}\lr{\alpha^{ij}\frac{\p{\rho^{\Psi}}}{\p x^j}}=S
,~~~~\langle\rho^{\Psi}\rangle=0~~~~{\rm in}~~D,\label{thmX2a}\\
&\rho^{\Psi}~~{\rm periodic}~~{\rm in}~~D, 
\end{align}\label{thmX2}
\end{subequations}
admits a unique solution 
${\rho}^{\Psi}\in H^1_{\rm per}\lr{D}$ (see for example \cite{Bensoussan}). 
Here, the notation $\rho^{\Psi}\lr{\mu,\nu}=\rho\lr{\mu,\nu,\Psi}$ stresses the fact that $\rho$ is being considered a function
of the angles $\lr{\mu,\nu}$ by fixing $\Psi\in U$. 
Let us briefly review the argument behind this result. 
Denote with $C^{\infty}_{\rm 
 per}\lr{D}=\left\{\rho^{\Psi}\in C^{\infty}\lr{\mathbb{R}^2};\langle\rho^{\Psi}\rangle=0,\rho^{\Psi}~~{\rm periodic}~~{\rm in}~~D\right\}$ the set of smooth functions  periodic in $D$ and with vanishing average. 
 Note that $C^{\infty}_{\rm per}\lr{D}=C^{\infty}\lr{\mathbb{R}^2}\cap H^1_{\rm per}\lr{D}$. 
 Then, the Hilbert  space $H^{1}_{\rm per}\lr{D}$ can be identified with the completion of $C^{\infty}_{\rm per}\lr{D}$ with respect to the $H^1$ norm. 
 Now observe that the weak formulation of \eqref{thmX2} is
 \begin{equation}
\lr{\rho^{\Psi},\psi}+\mc{F}_S\left[\psi\right]=\int_{D}\lr{\alpha^{ij}\frac{\p\psi}{\p x^i}\frac{\p\rho^{\Psi} }{\p x^j}+S\psi}\,d\mu d\nu=0~~~~\forall\psi\in H_{\rm per}^{1}\lr{D}.\label{wsol}
 \end{equation}
Indeed, if $\rho^{\Psi}\in C^2_{\rm per}\lr{D}$ with $C^{2}_{\rm per}\lr{D}=\left\{\rho^{\Psi}\in C^2\lr{\mathbb{R}^2};\langle\rho^{\Psi}\rangle=0,\rho^{\Psi}~~{\rm periodic}~~{\rm in}~~D\right\}$, the partial derivatives $\p\rho/\p x^i$, $i=1,2$, are periodic, and therefore integration by parts shows that $\rho^{\Psi}$ is a classical solution.
We also remark that if $\rho^{\Psi}$ is a weak solution in the sense of \eqref{wsol}, it can also be tested against any $\psi_0\in H^1_0\lr{D}$ since $\psi_0-\langle\psi_0\rangle /4\pi^2\in H^1_{\rm per}\lr{D}$ by periodic extension of $\psi_0$ to $\mathbb{R}^2$. The converse is however not true, since a solution $\rho^{\Psi}_0\in H^1_0\lr{D}$ of the standard  Dirichlet boundary value problem cannot be tested against functions $\psi\in H^1_{\rm per}\lr{D}$, i.e. $\lr{\rho^{\Psi}_0,\psi}+\mc{F}_S\left[\psi\right]\neq 0$ in general.

Next, note that the inner product
\begin{equation}
\lr{\rho^{\Psi},\psi}=\int_D\alpha^{ij}\frac{\p\rho^{\Psi}}{\p x^i}\frac{\p \psi}{\p x^j}\,d\mu d\nu,
\end{equation}
defines a norm $\Norm{\rho^{\Psi}}_{H^1_{\rm per}\lr{D}}=\lr{\rho^{\Psi},\rho^{\Psi}}^{1/2}$ in $H^1_{\rm per}\lr{D}$ 
due to the strict ellipticity of $\alpha^{ij}$, 
\begin{equation}
\lr{\rho^{\Psi},\rho^{\Psi}}\geq \lambda \Norm{\nabla_{\lr{\mu,\nu}}\rho^{\Psi}}_{L^2\lr{D}}^2\geq C\Norm{\rho^{\Psi}}_{H^1\lr{D}}^2,
\end{equation}
for some constant $C>0$ and 
where in the last passage we used the Poincar\'e inequality \cite{Necas}, the fact that $\langle\rho^{\Psi}\rangle=0$, and introduced the notation 
\begin{equation}
\begin{split}
&\Norm{\rho^{\Psi}}_{L^2\lr{D}}^2=\int_D\lr{\rho^{\Psi}}^2\,d\mu d\nu,\\&\Norm{\nabla_{\lr{\mu,\nu}}\rho^{\Psi}}_{L^2\lr{D}}^2=\int_D\left[\lr{\rho^{\Psi }_{\mu}}^2+\lr{\rho_{\nu}^{\Psi}}^2\right]\,d\mu d\nu,\\&
\Norm{\rho^{\Psi}}_{H^1\lr{D}}^2=\Norm{\rho^{\Psi}}_{L^2\lr{D}}^2+\Norm{\nabla_{\lr{\mu,\nu}}\rho^{\Psi}}_{L^2\lr{D}}^2.
\end{split}
\end{equation}
Hence, $H^1_{\rm per}\lr{D}$ is a Hilbert space with respect to the norm $\Norm{\cdot}_{H^1_{\rm per}\lr{D}}$. 
Finally, the linear functional 
\begin{equation}
\mc{F}_{S}\left[\psi\right]=\int_{D}S\psi\,d\mu d\nu\leq C\Norm{\psi}_{H^1_{\rm per}\lr{D}},
\end{equation}
is bounded, with $C>0$ a constant. Hence, the Riesz representation theorem guarantees the existence of a unique element $\rho^{\Psi}\in H^1_{\rm per }\lr{D}$ such that $F_{S}\left[\psi\right]=-\lr{\rho^{\Psi},\psi}$, which thus provides a weak solution of \eqref{thmX2}.
 

The construction above applies even if the origin of the cell $D$  
is shifted by an arbitrary amount in $\mathbb{R}^2$. 
Let $D'\subset\mathbb{R}^2$ denote the shifted cell and $\rho'^{\Psi}\in H^1_{\rm per}\lr{D'}=H^1_{\rm per}\lr{D}$ the corresponding solution. 
By interior regularity, 
any irregularity of the solution ${\rho'}^{\Psi}$ that may occur on the boundary $\p D'$ 
cannot affect the interior of the domain,
and it can be shown that the regularity of $\alpha^{ij}$ and $S$ is propagated to $\rho'^{\Psi}$. 
In particular, ${\rho'^{\Psi}\in C^{\infty}}\lr{D'}$ (see \cite{Evans}). 
Since we may take $D'\cap D\neq \emptyset$ and $\rho^{\Psi}=\rho'^{\Psi}$ by uniqueness, 
this also implies the regularity of the derivatives of $\rho^{\Psi}$ at the original cell boundary $\p D$, and thus their periodicity. 
We conclude that $\rho^{\Psi}\in C^{\infty}_{\rm per}\lr{D}$ and that all partial derivatives of any order of the function $\rho^{\Psi}$ are periodic functions in $D$.


We are now left with the task of showing that solutions of \eqref{thmX2}
corresponding to different values of $\Psi$ define a smooth function in the variable $\Psi$. 
To see this,  
it is convenient to 
introduce the linear differential operators 
\begin{equation}
L=\frac{\p}{\p x^i}\lr{\alpha^{ij}\frac{\p}{\p x^j}},~~~~L_{\Psi}=\frac{\p}{\p x^i}\lr{\alpha^{ij}_{\Psi}\frac{\p}{\p x^j}},
\end{equation}
where $L_{\Psi}={\p L}/{\p\Psi}$ and we used the fact that $\alpha^{ij}$ is smooth in the variable $\Psi$ to evaluate $\alpha^{ij}_{\Psi}=\p\alpha^{ij}/\p\Psi$. The first equation in \eqref{thmX2} thus takes the form $L\rho^{\Psi}=S$. 
Furthermore, the linear operator $L$ 
defines an invertible linear mapping from the function space $C^{\infty}_{\rm per}\lr{D}$ to itself (for $S\in C^{\infty}_{\rm per}\lr{D}$, $L\rho^{\Psi}=S$ admits a unique solution in $C^{\infty}_{\rm per}\lr{D}$).
Denoting with $L^{-1}$ the inverse, it follows that
\begin{equation}
0=\frac{\p\lr{LL^{-1}}}{\p\Psi}=L_{\Psi}L^{-1}+LL^{-1}_{\Psi},
\end{equation}
where $L_{\Psi}^{-1}=\p L^{-1}/\p \Psi$. 
Since $L^{-1}\lr{0}=0$, application of $L^{-1}$ to the equation above  
gives 
\begin{equation}
L_{\Psi}^{-1}=-L^{-1}L_{\Psi}L^{-1},\label{Lm1}
\end{equation}
so that the $\Psi$-derivative of the inverse operator $L^{-1}$ 
is expressed in terms of the operators $L^{-1}$ and $L_{\Psi}$.  
Higher order derivatives of the operator $L^{-1}$ can be determined by differentiating \eqref{Lm1} with respect to $\Psi$. 
We now consider $\rho$ as a function of the three variables $\lr{\mu,\nu,\Psi}$. 
From $\rho=L^{-1}S$, and observing that the quantity $S_{\Psi}-L_{\Psi}\rho$ belongs to $C^{\infty}_{\rm per}\lr{D}$ when intended as a function of $\mu,\nu$, we thus conclude that
\begin{equation}
\frac{\p\rho}{\p\Psi}=L_{\Psi}^{-1}S+L^{-1}S_{\Psi}=L^{-1}\lr{S_{\Psi}-L_{\Psi}\rho},~~~~\lr{\frac{\p\rho}{\p\Psi}}^{\Psi}\in C^{\infty}_{\rm per}\lr{D},\label{rhopsi1} 
\end{equation}
where $\lr{\p\rho/\p\Psi}^{\Psi}$ denotes the two variables function 
obtained by fixing $\Psi$ in $\p\rho/\p\Psi$. 
Similarly, $\p^2\rho/\p\Psi^2$ and higher order partial derivatives can be evaluated by repeatedly differentiating $\rho=L^{-1}S$ with respect to $\Psi$. Hence, for each $\Psi\in U$ derivatives of $\rho$ with respect to $\Psi$ of all order exist and belong to $C^{\infty}_{\rm per}\lr{D}$. 
It follows that the function $\rho$ is smooth in the variable $\Psi$, and 
therefore provides a unique solution $\rho\in 
C^{\infty}\lr{\mathbb{R}^2\times U}$ with period $D$ 
of the original boundary value problem \eqref{thmX1}
such that 
$\rho^{\Psi}\in C^{\infty}\lr{\mathbb{R}^2}\cap H^1_{\rm per}\lr{D}$. 
\end{proof}

We are now ready to prove theorem 1: 
\begin{proof}
By hypothesis, the function $\Psi$ is smooth and foliates the domain $\Omega$ with nested toroidal surfaces spanned by angle coordinates $\mu,\nu$. 
The smoothness of the derivatives of the curvilinear coordinate system $\lr{x^1,x^2,x^3}=\lr{\mu,\nu,\Psi}$ ensures that the
components $g_{\mu\mu}$, $g_{\mu\nu}$, $g_{\nu\nu}$, $g_{\Psi\mu}$, $g_{\nu\Psi}$, $g_{\Psi\Psi}$ 
of the metric tensor and the Jacobian $J$
are smooth functions in $\Omega$. Indeed, they can be expressed in terms of derivatives of the coordinates. For example
\begin{equation}
g_{\mu\mu}=\frac{g^{\nu\nu}g^{\Psi\Psi}-\lr{g^{\nu\Psi}}^2}{J^2}=\frac{\abs{\nabla\nu}^2\abs{\nabla\Psi}^2-\lr{\nabla\nu\cdot\nabla\Psi}^2}{\lr{\nabla\mu\cdot\nabla\nu\cp\nabla\Psi}^2}.
\end{equation}
Hence, the two-dimensional matrix $A$ 
defined in \eqref{A} has smooth components
$a^{ij}$, $i,j=1,2$, in $\Omega$. 
Furthermore, as shown in the previous section
the matrix $A$ is symmetric and positive definite, and such that the corresponding differential operator $\p_i\lr{Ja^{ij}\p_j}$ in \eqref{redeq0} 
is strictly elliptic on each $\Psi$ contour (notice that by hypothesis $J\geq J_m>0$ for some positive constant $J_m$ so that the strict ellipticity of $a^{ij}$ implies the strict ellipticity of $Ja^{ij}$). 
Recalling that the composition of smooth functions is smooth, it is now clear that the hypothesis of lemma 1 are satisfied with $\alpha^{ij}=Ja^{ij}$ and  
source term $S$ given by  
\begin{equation}
S=\frac{\p}{\p\nu}\lr{Jg_{\mu\nu}}-\frac{\p}{\p\mu}\lr{Jg_{\nu\nu}}. 
\end{equation}
In particular, observe that $S$ is smooth.  
Let $\rho\in 
C^{\infty}\lr{\mathbb{R}^2\times U}$ 
denote the periodic classical solution of equation \eqref{redeq0} obtained in accordance with lemma 1. Evidently, $\rho\in C^{\infty}\lr{\Omega}$ as well.
Setting $\Theta=\mu+\rho$ and recalling equation \eqref{wcurlw}, it follows that the vector field
\begin{equation}
\bol{w}=\nabla\Psi\cp\nabla\Theta=J\lr{\Theta_{\mu}\p_{\nu}-\Theta_{\nu}\p_{\mu}}=J\p_{\nu}+\nabla\Psi\cp\nabla\rho,\label{wF}
\end{equation}
is a solution $\bol{w}\in C^{\infty}\lr{\Omega}$
of system \eqref{eq0}. 
To see this, 
first recall that $\rho$ is a smooth solution of equation \eqref{redeq0}, and thus the vector field \eqref{wF} fulfills equation \eqref{eq1X} in the hollow torus $\Omega$. 
Since 
$\nabla\cdot\bol{w}=0$, the vector field \eqref{wF} therefore solves system \eqref{eq0} in  $\Omega$. 
Furthermore, the vector field \eqref{wF} is non-vanishing since 
\begin{equation}
\begin{split}
\langle\Theta_{\mu}\rangle=\int_{D}\Theta_{\mu}\,d\mu d\nu=\int_{D}\lr{1+\rho_{\mu}}\,d\mu d\nu=4\pi^2.
\end{split}
\end{equation}
Recalling that the partial derivative  $\Theta_{\mu}$ is smooth and that    
$\bol{w}=J\lr{\Theta_{\mu}\p_{\nu}-\Theta_{\nu}\p_{\mu}}$ it follows that $\bol{w}\neq\bol{0}$ in some open set within $\Omega$. 
It may happen however that the solution $\bol{w}$ is a curl-free (vacuum) solution 
$\nabla\cp\bol{w}=\bol{0}$, 
or a Beltrami field $\nabla\cp\bol{w}=\hat{h}\bol{w}$ for some proportionality coefficient $\hat{h}\lr{\bol{x}}\neq 0$. 
Nevertheless, denoting with
$f\lr{\Psi}\neq 0$ any smooth function of the variable $\Psi$ such that $\p f/\p\Psi\neq 0$, 
it readily follows that in such scenario 
the vector field $\bol{w}'=f\lr{\Psi}\bol{w}$
is a nontrivial solution of \eqref{eq0}. Indeed, 
recalling that by construction $\bol{w}\cdot\nabla\Psi=0$, one has
\begin{equation}
\lr{\nabla\cp\bol{w}'}\cp\bol{w}'=-\frac{1}{2}\frac{\p f^2}{\p\Psi}\bol{w}^2\nabla\Psi\neq\bol{0},~~~~\nabla\cdot\bol{w}'=\frac{\p f}{\p\Psi}\nabla\Psi\cdot\bol{w}=0.
\end{equation}
\end{proof}


\begin{remark}
In the original formulation of the problem \eqref{eq1}, the domain $\Omega$ is a torus. However, the result of theorem 1 applies to a hollow torus. 
For the solution $\bol{w}$ of theorem 1 to hold in the hollow region as well, the vector field $\bol{w}$ must be well defined when approaching the toroidal axis. 
This is often the case, as it will be shown in the example constructed in section 6. 
\end{remark}

\begin{remark}
In the study of the vorticity equation for fluid flows over two-dimensional surfaces parametrized by $\Psi$ and embedded in three-dimensional Euclidean space, the relationship between the component of the vorticity $\omega^{\Psi}=\bol{\omega}\cdot\nabla\Psi$ and the stream function $\Theta$ is precisely $\nabla\cdot\left[\nabla\Psi\cp\lr{\nabla\Theta\cp\nabla\Psi}\right]=-\omega^{\Psi}$ (see \cite{SY22}). The result of lemma 1 thus implies that one can solve for the stream function $\Theta$ knowing the vorticity $\omega^{\Psi}$. Notice in particular that the topology of the level sets of $\Psi$ does not need to be toroidal.  
\end{remark}

\section{Example of numerical solution}

The aim of this section is to provide a numerical example of solution of equation \eqref{redeq0}. This example will also clarify the role played by periodic boundary conditions in ensuring the regularity of the solution $\bol{w}$ of \eqref{eq1} and its derivatives. 
To this end, we consider a family of toroidal surfaces corresponding to level sets of the function 
\begin{equation}
\Psi=\frac{1}{2}\lr{r-r_0}^2+\frac{1}{2}\mc{E}\lr{z-h}^2,\label{Psi}
\end{equation}
with $h=h\lr{\varphi,z}$, $r_0>0$ and $\mc{E}>0$ real constants, 
and $\lr{r,\varphi,z}$ cylindrical coordinates. 
The constant $r_0$ represents the major radius of the torus.  
When $\mc{E}=1$ and $h=0$, 
level sets of \eqref{Psi} 
correspond to axially symmetric toroidal surfaces enclosing a toroidal volume $\Omega$ with circular cross-section $\Sigma_{\varphi}=\left\{\bol{x}\in\Omega:\varphi=\varphi_0\in\left[0,2\pi\right)\right\}$. 
If $\mc{E}\neq 1$ the cross-sections $\Sigma_{\varphi}$ 
depart from circles, while a non-zero $h$ can be regarded as a displacement of the toroidal axis in the vertical direction. 
Figure \ref{fig2} shows 
examples of toroidal surfaces obtained as level sets of \eqref{Psi}. 
\begin{figure}[h]
\hspace*{-0cm}\centering
\includegraphics[scale=0.5]{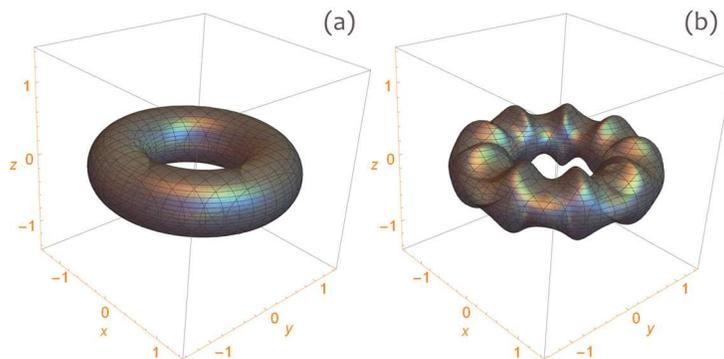}
\caption{\footnotesize (a) Axially symmetric torus corresponding to the level set $\Psi=0.08$ with $r_0=1$, $\mc{E}=1$, and $h=0$ in equation \eqref{Psi}. (b) Torus corresponding to the level set $\Psi=0.08$ with $r_0=1$, $\mc{E}=1.6$, $h=0.3 z\sin\lr{9\varphi}$ in equation \eqref{Psi}.}
\label{fig2}
\end{figure}
Notice that an appropriate choice of the function 
$h$ breaks the rotational (axial) symmetry of the surface. More generally, it is possible to construct toroidal surfaces that are not invariant under continuous Euclidean isometries (combinations of translations and rotations). 
For example, setting $h=\epsilon z \sin\lr{m\varphi}$ with $\epsilon>0$ a real constant 
and $m\neq 0$ an integer, 
the corresponding toroidal surface is not invariant under continuous Euclidean isometries. Indeed, the Lie-derivative 
\begin{equation}
\mf{L}_{\bol{a}+\bol{b}\cp\bol{x}}\Psi=\lr{\bol{a}+\bol{b}\cp\bol{x}}\cdot\nabla\Psi,\label{Lie1}
\end{equation}
where $\bol{a}+\bol{b}\cp\bol{\xi}$, $\bol{a},\bol{b}\in\mathbb{R}^3$, is the generator of continuous Euclidean isometries in $\mathbb{R}^3$, vanishes only if $\bol{a}=\bol{b}=\bol{0}$. 
This can be verified by evaluating \eqref{Lie1} 
on the planes $z=0$, $\varphi=0$, and $\varphi=\pi/2$, which respectively give the conditions $a_x=a_y=0$, $b_x=b_z=0$, and $b_y=a_z=0$. 
Therefore, the example shown in figure \ref{fig2}(b) is not invariant under continuous Euclidean isometries.   

In order to construct a solution $\bol{w}$ of system \eqref{redeq0}, we must now define angle coordinates $\mu,\nu\in [0,2\pi)$ spanning the toroidal surfaces $\Psi$. 
In particular, we consider the curvilinear coordinates $\lr{\mu,\nu,\Psi}=\lr{\varphi,\vartheta,\Psi}$ with 
\begin{equation}
\vartheta=\arctan\lr{\frac{z}{r-r_0}},
\end{equation}
the poloidal angle. 
The contravariant components of the metric tensor can be evaluated to be
\begin{equation}
\begin{split}
g^{\varphi\varphi}&=\frac{1}{r^2},~~~~g^{\varphi\vartheta}=0,~~~~g^{\vartheta\vartheta}=\frac{1}{z^2+\lr{r-r_0}^2},~~~~g^{\varphi\Psi}=-\mc{E}\frac{z-h}{r^2}h_{\varphi},\\
g^{\vartheta\Psi}&=\frac{r-r_0}{z^2+\lr{r-r_0}^2}\left[
\mc{E}\lr{z-h}\lr{1-h_z}-z
\right],~~~~g^{\Psi\Psi}=\lr{r-r_0}^2+\mc{E}^2\lr{z-h}^2\left[\lr{1-h_z}^2+\frac{h_{\varphi}^2}{r^2}\right].
\end{split}
\end{equation}
The covariant components are
\begin{equation}
\begin{split}
g_{\varphi\varphi}&=\frac{g^{\vartheta\vartheta}g^{\Psi\Psi}-\lr{g^{\vartheta\Psi}}^2}{J^2},~~~~g_{\varphi\vartheta}=\frac{g^{\vartheta\Psi}g^{\varphi\Psi}}{J^2},~~~~g_{\vartheta\vartheta}=\frac{g^{\varphi\varphi}g^{\Psi\Psi}-\lr{g^{\varphi\Psi}}^2}{J^2},\\g_{\varphi\Psi}&=-\frac{g^{\varphi\Psi}g^{\vartheta\vartheta}}{J^2},~~~~
g_{\vartheta\Psi}=-\frac{g^{\varphi\varphi}g^{\vartheta\Psi}}{J^2},~~~~g_{\Psi\Psi}=\frac{g^{\varphi\varphi}g^{\vartheta\vartheta}}{J^2}.
\end{split}\label{covg}
\end{equation}
We also have
\begin{equation}
J=\nabla\varphi\cdot\nabla\vartheta\cp\nabla\Psi=\frac{\lr{r-r_0}^2+\mc{E}z\lr{1-h_z}\lr{z-h
    }}{r\left[z^2+\lr{r-r_0}^2\right]}.\label{Jac}
\end{equation}
Next, let us consider a vertical axial displacement 
$h=\epsilon z\sin\lr{m\varphi}$. 
In this case, the inverse coordinate transformation reads  
\begin{equation}
z^2=\frac{2\Psi\sin^2\vartheta}{\cos^2\vartheta+\mc{E}\left[1-\epsilon\sin\lr{m\varphi}\right]^2\sin^2\vartheta},~~~~\lr{r-r_0}^2=\frac{2\Psi\cos^2\vartheta}{\cos^2\vartheta+\mc{E}\left[1-\epsilon\sin\lr{m\varphi}\right]^2\sin^2\vartheta}.\label{inv}
\end{equation}
Using \eqref{inv}, the metric coefficients \eqref{covg} and the Jacobian \eqref{Jac}
can be expressed explicitly as functions of $\lr{\mu,\nu,\Psi}$. 
Hence, we may attempt to solve equation \eqref{redeq0} in the doubly periodic domain $D$.
If a periodic solution $\rho$ with periodic derivatives could be found, 
the corresponding vector field $\bol{w}=\nabla\Psi\cp\nabla\Theta$ with $\Theta=\mu+\rho$ would provide 
the desired solution in $\Omega$. 
However, this task is not trivial, since the space of solutions is effectively restricted to functions $\rho=\sum_{m,n}c_{mn}\lr{\Psi}e^{{\rm i}\lr{m\mu+n\nu}}$ that are represented by a convergent Fourier series in the variables $\mu$ and $\nu$.
Nevertheless, numerical solutions can be obtained
in a rather straightforward fashion by 
sacrificing 
the continuity of the partial derivatives  $\rho_{\mu}$ and $\rho_{\nu}$ (and thus the continuity of $\bol{w}$ and $\nabla\cp\bol{w}$, recall \eqref{wcurlw}) on the points $\bol{x}\in\Omega$ corresponding to $\lr{\mu,\nu}\in\p D$. Indeed, 
coupling equation \eqref{redeq0} with 
Dirichlet boundary conditions
\begin{equation}
\rho\lr{0,\nu,\Psi}=\rho\lr{2\pi,\nu,\Psi}=\rho\lr{\mu,0,\Psi}=\rho\lr{\mu,2\pi,\Psi}=0,\label{Dbc}
\end{equation}
results in a usual elliptic problem that can be approached with standard numerical tools. 
The corresponding solution $\rho$ will be periodic in the variables $\mu$ and $\nu$, although only 
the partial derivative of $\rho$ tangential to the boundary $\p D$ will be periodic, while the
normal component will not. 
For completeness, we also note that the regularity of the function $\rho$ at the boundary $\p D$ is 
obstructed by the corners of the square domain $D$, 
which give $\rho\in C^{1,\alpha}(\bar{D})$ with H\"older coefficient $0<\alpha<1$ (see \cite{Azzam,Azzam2}). 


Figure \ref{fig3} shows two examples of numerical solution of equation \eqref{redeq0} with Dirichlet boundary conditions \eqref{Dbc}. 
Notice that while $\rho$ is periodic in $\mu$ and $\nu$, $\rho_{\mu}$ is periodic only between $\nu=0$ and $\nu=2\pi$, while it takes different values at $\mu=0$ and $\mu=2\pi$. Analogous considerations apply to $\rho_{\nu}$, $\Theta_{\mu}=1+\rho_{\mu}$, and $\Theta_{\nu}=\rho_{\nu}$. 
As already explained at the end of section 3, 
this implies that the corresponding vector fields   $\bol{w}$ and $\nabla\cp\bol{w}$, which are given by \eqref{wcurlw}, will exhibit discontinuities at  the points $\bol{x}\in\Omega$ corresponding to $\lr{\mu,\nu}\in\p D$. 
\begin{figure}[h]
\hspace*{-0cm}\centering
\includegraphics[scale=0.5]{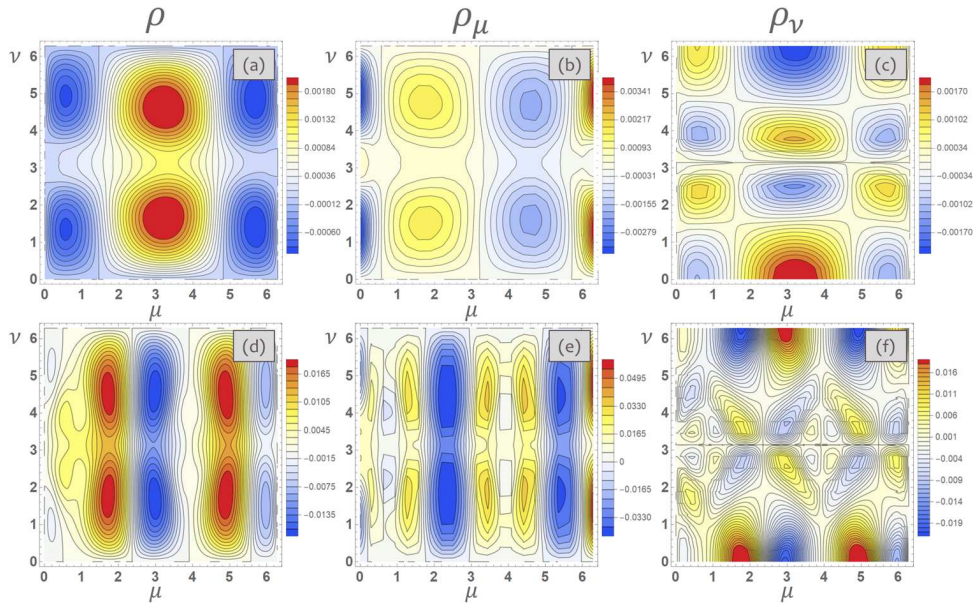}
\caption{\footnotesize (a), (b), and (c): numerical solution $\rho$ of equation \eqref{redeq0} with Dirichlet boundary conditions \eqref{Dbc} and its partial derivatives $\rho_{\mu}$ and $\rho_{\nu}$ for $r_0=1$, $m=1$,   $\mc{E}=1.6$, $h=\epsilon z\sin\lr{m\varphi}$, and $\epsilon=0.03$ on the toroidal surface $\Psi=0.16$, with $\Psi$ given by equation \eqref{Psi}. (d), (e), and (f): numerical solution $\rho$ of equation \eqref{redeq0} with Dirichlet boundary conditions \eqref{Dbc} and its partial derivatives $\rho_{\mu}$ and $\rho_{\nu}$ for $r_0=1$, $m=2$, $\mc{E}=1.6$, $h=\epsilon z\sin\lr{m\varphi}$, and $\epsilon=0.3$ on the toroidal surface $\Psi=0.08$, with $\Psi$ given by equation \eqref{Psi}.}
\label{fig3}
\end{figure}
Finally, we remark that, in principle, the smooth vector field $\bol{w}$ constructed in theorem 1 could be numerically computed by expanding in Fourier series each term in equation \eqref{redeq0} and by solving for the Fourier coefficients of the solution $\rho$.  












\section{Example of smooth solution and relation with anisotropic magnetohydrodynamics}
In this section we construct an example of smooth solution $\bol{w}\in C^{\infty}\lr{\Omega}$ of equation \eqref{eq1} such that $\lr{\nabla\cp\bol{w}}\cp\bol{w}\neq\bol{0}$
in a toroidal domain $\Omega$.
To this end, the following observation is useful:

\begin{proposition}
Let $\Omega\subset\mathbb{R}^3$ be a 
toroidal volume with boundary $\p\Omega$ foliated by toroidal surfaces corresponding to level sets of a function $\Psi\in C^1(\bar{\Omega})$.  
Let $\bol{\xi}\in L^2_H\lr{\Omega}$ be a harmonic vector field in $\Omega$, with 
\begin{equation}
L^2_H\lr{\Omega}=\left\{\bol{\xi}\in L^2\lr{\Omega};\nabla\cp\bol{\xi}=\bol{0},\nabla\cdot\bol{\xi}=0,\bol{\xi}\cdot\bol{n}=0\right\},
\end{equation}
where $\bol{n}$ denotes the unit outward normal to $\p\Omega$. 
Further assume that $\bol{\xi}$ is foliated by $\Psi$, that is
\begin{equation}
\bol{\xi}\cdot\nabla\Psi=0~~~~{\rm in}~~\Omega.
\end{equation}
Then, the vector field $\bol{w}\in H^1_{\sigma\sigma}\lr{\Omega}$ defined as 
\begin{equation}
\bol{w}=f\lr{\Psi}\bol{\xi},
\end{equation}
where $f$ is any $C^1(\bar{\Omega})$ function of $\Psi$, 
is a nontrivial solution of \eqref{eq1} in $\Omega$ such that
\begin{equation}
\lr{\nabla\cp\bol{w}}\cp\bol{w}=-\frac{1}{2}\frac{\p f^2}{\p\Psi}\abs{\bol{\xi}}^2\nabla\Psi,~~~~\nabla\cdot\bol{w}=0.
\end{equation} 
Here, 
\begin{equation}
H^1_{\sigma\sigma}\lr{\Omega}=\left\{\bol{w}\in L^2_{\sigma}\lr{\Omega};\nabla\cp\bol{w}\in L^2_{\sigma}\lr{\Omega}\right\},
\end{equation}
\begin{equation}
L^2_{\sigma}\lr{\Omega}=\left\{\bol{w}\in L^2\lr{\Omega};\nabla\cdot\bol{w}=0,\bol{w}\cdot\bol{n}=0\right\}.
\end{equation}
\end{proposition}
The proof of the above statement 
follows by direct evaluation of
$\lr{\nabla\cp\bol{w}}\cp\bol{w}$ and $\nabla\cdot\bol{w}$. 
Recall that the dimension of the linear space $L^2_H\lr{\Omega}$ is given by the genus of $\p\Omega$. 
For a toroidal surface with genus 1 
the space of harmonic vector fields $\bol{\xi}\in L^2_H\lr{\Omega}$ is therefore 1-dimensional. 
We refer the reader to
\cite{Schwarz,Pfeff} for additional details
on harmonic vector fields,  
which arise in the context of    
Hodge decomposition of differential forms.

Proposition 1 suggests that 
solutions of equations \eqref{eq1}
can be obtained by identifying harmonic vector fields foliated by toroidal surfaces. The prototypical example of such vector field is the gradient $\bol{\xi}_0=\nabla\varphi$ of the toroidal angle $\varphi$, which is tangential 
to axially symmetri tori parametrized by $\Psi_0=\frac{1}{2}\left[\lr{r-r_0}^2+z^2\right]$. 
To break axial symmetry, we proceed as follows. 
First, we perturb the toroidal angle according to 
\begin{equation}
\eta=\varphi+\epsilon \sigma, 
\end{equation}
where $\sigma$ is chosen to be a harmonic function  
so that the vector field
\begin{equation}
\bol{\xi}_{\epsilon}=\nabla\eta=\nabla\varphi+\epsilon\nabla\sigma\in L^2_H\lr{\Omega},
\end{equation}
is a harmonic vector field in a toroidal domain $\Omega$ whose precise shape has yet to be determined. In particular, for sufficiently small $\epsilon>0$, we  expect to find a function $\Psi_{\epsilon}$ 
such that $\bol{\xi}_{\epsilon}\cdot\nabla\Psi_{\epsilon}=0$ and the level sets of $\Psi_{\epsilon}$ define toroidal surfaces. Indeed, the limit $\epsilon\rightarrow 0$ corresponds to the axially symmetric case of the vector field $\bol{\xi}_0=\nabla\varphi$ tangential to contours of $\Psi_0$. A simple choice for the perturbation is the harmonic function  $\sigma=r^{m}\cos \lr{m\varphi}$, $m\in\mathbb{Z}$, $m\neq 0$. For example, take $m=1$ so that $\sigma=r\cos\varphi=x$. Then, 
the following orthogonality condition must be solved for $\Psi_{\epsilon}$, 
\begin{equation}
\bol{\xi}_{\epsilon}\cdot\nabla\Psi_{\epsilon}=\frac{1}{r^2}\lr{1-\epsilon r\sin\varphi}\frac{\p\Psi_{\epsilon}}{\p\varphi}+\epsilon\cos\varphi\frac{\p\Psi_{\epsilon}}{\p r}=0.
\end{equation}
One can verify that a solution is given by the function
\begin{equation}
\Psi_{\epsilon}=\frac{1}{2}\left[\lr{re^{-\epsilon y}-r_0}^2+z^2\right],\label{psiepx}
\end{equation}
where $r_0>0$ is a real constant.
Contours of $\Psi_{\epsilon}$ define toroidal surfaces as shown in figure \ref{fig4}. 
Notice also that both $\Psi_{\epsilon}$ and the vector field
\begin{equation}
\bol{w}=f\lr{\Psi_{\epsilon}}\bol{\xi}_{\epsilon},\label{wxi}
\end{equation}
are smooth within the toroidal volume $\Omega$ enclosed by $\Psi_{\epsilon}$
for a suitable choice of $f\lr{\Psi_{\epsilon}}$.  In addition, 
\begin{equation}
\lr{\nabla\cp\bol{w}}\cp\bol{w}=-\frac{1}{2}\frac{\p f^2}{\p\Psi_{\epsilon}}
\frac{\lr{1-\epsilon y}^2+\epsilon^2x^2}{r^2}
\nabla\Psi_{\epsilon},~~~~\nabla\cdot\bol{w}=0.
\end{equation}
A plot of the vector field \eqref{wxi} for $f=\exp\left\{\Psi_{\epsilon}/2\right\}$ is given in figure \ref{fig4}. 
It should be noted that both $\bol{w}$ and $\nabla\cp\bol{w}$ diverge when $r\rightarrow 0$, a fact that makes the constructed solution unphysical in $\mathbb{R}^3$ (an infinite current $\nabla\cp\bol{w}$ would be needed on the vertical axis to sustain such a magnetic field $\bol{w}$). Nevertheless, this divergence 
is not worrisome as it is analogous to the 
divergence of the magnetic field $\bol{B}\propto\nabla\varphi$ generated by
a straight current flowing along the vertical axis.
Notice also that $\Psi_{\epsilon}$ and $\abs{\nabla\Psi_{\epsilon}}$ diverge at large distances from the origin, but the corresponding divergences in $\bol{w}$ and $\nabla\cp\bol{w}$ can be suppressed by appropriate choice of $f$.   
\begin{figure}[h]
\hspace*{-0cm}\centering
\includegraphics[scale=0.5]{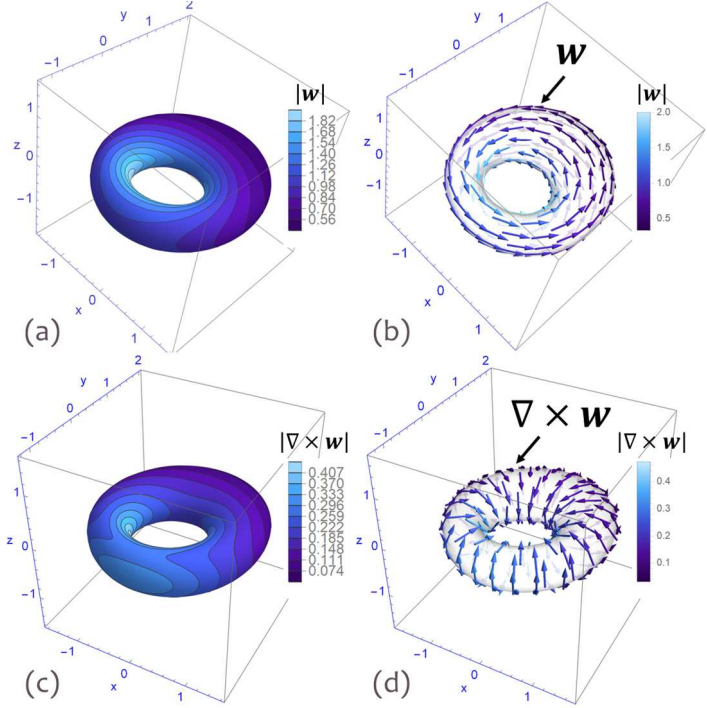}
\caption{\footnotesize (a) Contour plot of $\abs{\bol{w}}$ over the level set $\Psi_{\epsilon}=0.08$. 
(b) Vector plot of $\bol{w}$ over the level set $\Psi_{\epsilon}=0.08$. 
(c) Contour plot of $\abs{\nabla\cp\bol{w}}$ over the level set $\Psi_{\epsilon}=0.08$. 
(d) Vector plot of $\nabla\cp\bol{w}$ over the level set $\Psi_{\epsilon}=0.08$. 
In (a), (b), (c) and (d) $\Psi_{\epsilon}$ is defined by equation \eqref{psiepx} with $r_0=1$ and $\epsilon=0.18$, and 
$\bol{w}$ is defined by equation \eqref{wxi} with $f=\exp\left\{\Psi_{\epsilon}/2\right\}$ and  $\bol{\xi}_{\epsilon}=\nabla\lr{\varphi+\epsilon r\cos\varphi}$. Observe that the vector field $\bol{w}$ shown here is a solution of \eqref{eq1} such that both the bounding surface $\Psi_{\epsilon}$ and the vector field $\bol{w}$ are not invariant under continuous Euclidean isometries.}
\label{fig4}
\end{figure}

Let us now verify that 
both $\Psi_{\epsilon}$ and the vector field $\bol{w}$ defined in equation \eqref{wxi}
are not endowed with continuous Euclidean isometries. 
Following the same procedure 
of section 2, we must determine 
constant vectors $\bol{a},\bol{b}\in\mathbb{R}^3$ 
such that the Lie derivative below vanishes, 
\begin{equation}
\begin{split}
\mf{L}_{\bol{a}+\bol{b}\cp\bol{x}}\Psi_{\epsilon}=&
e^{-\epsilon y}\frac{r e^{-\epsilon y}-r_0}{r}\left[\lr{a_x+b_yz-b_zy}
x+\lr{a_y+b_zx-b_x z}\lr{y-\epsilon r^2}
\right]\\&
+\lr{a_z+b_xy-b_yx}z=0
\end{split}
\end{equation}
Considering the section $\varphi=\pi/2$, 
one obtains the condition
\begin{equation}
\mf{L}_{\bol{a}+\bol{b}\cp\bol{x}}\Psi_{\epsilon}=
\lr{a_y-b_x z}e^{-\epsilon r}\lr{r e^{-\epsilon r}-r_0}\lr{1-\epsilon r}+z\lr{a_z+b_x r}=0.
\end{equation}
When $z=0$, the expression above holds only if $a_y=0$. 
Then, for $z\neq 0$ it follows that $a_z=b_x=0$ as well. 
Similarly, at $\varphi=0$, one has
\begin{equation}
\mf{L}_{\bol{a}+\bol{b}\cp\bol{x}}\Psi_{\epsilon}=
\lr{r-r_0}\lr{a_x+b_yz-\epsilon b_z r^2}-b_y rz=0.
\end{equation}
When $z=0$, this quantity vanishes for arbitrary $r$ provided that $a_x=b_z=0$. 
We therefore conclude that $b_y=0$ as well, and $\Psi_{\epsilon}$ is not invariant under continuous Euclidean isometries. 
To ascertain that the vector field  $\bol{w}=f\lr{\Psi_{\epsilon}}\bol{\xi}_{\epsilon}$ is not invariant under the same class of transformations, it is sufficient to study the symmetry of its modulus. 
Indeed, by standard vector identities 
\begin{equation}
\bol{w}\cdot\mf{L}_{\bol{a}+\bol{b}\cp\bol{x}}\bol{w}=\frac{1}{2}\mf{L}_{\bol{a}+\bol{b}\cp\bol{x}}\bol{w}^2.
\end{equation}
Hence, if the Lie-derivative of the modulus $\mf{L}_{\bol{a}+\bol{b}\cp\bol{x}}\bol{w}^2$ does not vanish, the Lie-derivative of the vector field $\mf{L}_{\bol{a}+\bol{b}\cp\bol{x}}\bol{w}$ does not vanish as well. Next, observe that 
\begin{equation}
\mf{L}_{\bol{a}+\bol{b}\cp\bol{x}}\bol{w}^2=
\bol{w}^2\left[\lr{\bol{a}+\bol{b}\cp\bol{x}}\cdot\nabla \log f^2+\lr{\bol{a}+\bol{b}\cp\bol{x}}\cdot\nabla\log\abs{\bol{\xi_{\epsilon}}}^2\right]=0.
\end{equation}
Consider, for example, the case $f^2=\exp\left\{\Psi_{\epsilon}\right\}$. 
We have
\begin{equation}
\begin{split}
\mf{L}_{\bol{a}+\bol{b}\cp\bol{x}}\bol{w}^2=&
\bol{w}^2\left\{
\lr{a_x+b_yz-b_zy}\left[
xe^{-\epsilon y}\frac{re^{-\epsilon y}-r_0}{r}+2x\lr{\frac{\epsilon^2}{\lr{1-\epsilon y}^2+\epsilon^2 x^2}-\frac{1}{r^2}}
\right]\right.\\
&+\lr{a_y+b_zx-b_xz}\left[e^{-\epsilon y}\frac{re^{-\epsilon y}-r_0}{r}\lr{y-\epsilon r^2}-2\left(\frac{y}{r^2}+\epsilon\frac{1-\epsilon y}{\lr{1-\epsilon y}^2+\epsilon^2 x^2}\right)\right]
\\ &\left.
\lr{a_z+b_xy-b_yx}z
\right\}=0.
\end{split}
\end{equation}
On the section $\varphi=\pi/2$, we obtain the condition
\begin{equation}
\mf{L}_{\bol{a}+\bol{b}\cp\bol{x}}\bol{w}^2=\bol{w}^2\left\{
\lr{a_y-b_xz}\left[e^{-\epsilon r}\lr{r e^{-\epsilon r}-r_0}\lr{1-\epsilon r}-2\lr{\frac{1}{r}+\frac{\epsilon}{{1-\epsilon r}}}\right]
+\lr{a_z+b_xr}z
\right\}.
\end{equation}
Setting $z=0$ leads to $a_y=0$. 
Since $r$ and $z$ are not constants, 
it thus follows that 
the equation above can be satisfied only if $a_z=b_x=0$ as well.
Next, on the section $\varphi=0$ we have
\begin{equation}
\mf{L}_{\bol{a}+\bol{b}\cp\bol{x}}\bol{w}^2=\bol{w}^2\left\{
\lr{a_x+b_yz}\left[r-r_0+2r\lr{\frac{\epsilon^2}{1+\epsilon^2 r^2}-\frac{1}{r^2}}
\right]-\epsilon b_zr\left[ r\lr{r-r_0}+\frac{2}{1+\epsilon^2 r^2}\right]
-b_yrz
\right\}.
\end{equation}
Considering the case $z=0$, it follows that $a_x=b_z=0$. But then $b_y=0$ as well. 
Hence, for $f^2=\exp\left\{\Psi_{\epsilon}\right\}$, the modulus  $\bol{w}^2$ (and thus $\bol{w}$) is not invariant under continuous Euclidean isometries.  

The solution constructed above can be generalized to a wider class of solutions with the aid of two-dimensional harmonic conjugate functions $P\lr{x,y}$ and $Q\lr{x,y}$ such that $\p P/\p x=\p Q/\p y$ and $\p P/\p y=-\p Q/\p x$. Explicitly, we can define the family of solutions $\bol{w}=f\lr{\Psi_{\epsilon}}\bol{\xi}_{\epsilon}$ to \eqref{eq1} with 
\begin{subequations}
\begin{align}
\bol{\xi}_{\epsilon}=&\nabla\left[\varphi+\epsilon P\lr{x,y}
\right],\\
\Psi_{\epsilon}=&\frac{1}{2}\left\{\left[re^{-\epsilon Q\lr{x,y}}
-r_0\right]^2+z^2e^{-\epsilon S\lr{z}}\right\},
\end{align}
\end{subequations}\label{eq88}
where $S\lr{z}$ is a function of $z$. For example, setting $P=e^{mx}\cos\lr{my}$, $Q=e^{mx}\sin\lr{my}$, $S=2\sin{z}$, with $m\in\mathbb{R}$, $m\neq 0$,  generates solutions $\bol{w}=f\lr{\Psi_{\epsilon}}\bol{\xi}_{\epsilon}$ of \eqref{eq1} without continuous Euclidean isometries. We also remark that such solutions do not possess discrete Euclidean isometries (reflections) as well. To see this, first note
that $\Psi_{\epsilon}$ is no longer invariant under the transformation $z\rightarrow -z$. Invariance under other reflections can be excluded as follows. Let $\bol{n}=\lr{n_x,n_y,n_z}\in\mathbb{R}^3$, $\bol{n}\neq\bol{0}$, denote the unit normal of a plane corresponding to a level set of the function  $\zeta=\bol{n}\cdot\bol{x}$, i.e. $\bol{n}=\nabla \zeta$ with $\bol{n}^2=1$. 
Next, choose $\bol{t},\bol{v}\in\mathbb{R}^3$, with $\bol{t}\cdot\bol{v}=\bol{t}\cdot\bol{n}=\bol{v}\cdot\bol{n}=0$, $\bol{t}^2=\bol{v}^2=1$,   
define $\eta=\bol{t}\cdot\bol{x}$, $\theta=\bol{v}\cdot\bol{x}$, and perform the change of coordinates $\lr{x,y,z}\rightarrow\lr{\zeta,\eta,\theta}$. Notice that the set $\lr{\zeta,\eta,\theta}$ is orthonormal by construction.  
A function $\Psi_{\epsilon}\lr{\zeta,\eta,\zeta}$ is endowed with a reflection symmetry $R_{\zeta}$ by the plane $\zeta$ provided that
\begin{equation}
\Psi_{\epsilon}\lr{\zeta,\eta,\theta}=R_\zeta \Psi_{\epsilon}\lr{\zeta,\eta,\theta}=
\Psi_{\epsilon}\lr{-\zeta,\eta,\theta}.\label{rsym}
\end{equation}
Using the fact that the Cartesian coordinates $\lr{x,y,z}$ are linear functions of the new coordinates $\lr{\zeta,\eta,\theta}$, one can verify that there exist no nontrivial choice of the vector $\bol{n}$ such that \eqref{rsym} is satisfied. 
Furthermore, the function $\Psi_{\epsilon}$ cannot be  invariant under combinations  of continuous and discrete Euclidean isometries, because  after a reflection $R_{\zeta}$ one can always define a new set of Cartesian coordinates $\bol{x}'=\lr{R_{\zeta}x,R_{\zeta}y,R_{\zeta}z}$ that preserve the functional form of $\Psi_{\epsilon}$, i.e. $R_{\zeta}\Psi_{\epsilon}=\Psi_{\epsilon}\lr{x\rightarrow R_{\zeta}x,y\rightarrow R_{\zeta}y,z\rightarrow R_{\zeta}z}$, implying that $\mf{L}_{\bol{a}+\bol{b}\times\bol{x}'}\Psi_{\epsilon}=\bol{0}$ if and only if $\bol{a}=\bol{b}=\bol{0}$. 
 
It should be noted that the vector fields \eqref{wxi}  constructed above can be regarded as steady solutions of anisotropic magnetohydrodynamics,
\begin{equation}
\lr{\nabla\cp\bol{w}}\cp\bol{w}=\nabla\cdot\Pi,~~~~\nabla\cdot\bol{w}=0~~~~{\rm in}~~\Omega,\label{eq3}
\end{equation}
where the Cartesian components of the pressure tensor $\Pi$ are given by \cite{SatoQS1,Grad67}
\begin{equation}
\Pi^{ij}=\lr{P-\frac{1}{2}\gamma \bol{w}^2}\delta^{ij}+\gamma w^i w^j,~~~~i,j=1,2,3,
\end{equation}
with $P$ a reference pressure field
and $\gamma$ the pressure anisotropy.  
Notice that equation \eqref{eq2} is recovered for $\gamma=0$. The first equation in \eqref{eq3} expressing anisotropic force balance can be written as
\begin{equation}
\lr{1-\gamma}\lr{\nabla\cp\bol{w}}\cp\bol{w}=\nabla P-\frac{1}{2}\bol{w}^2\nabla\gamma+\lr{\bol{w}\cdot\nabla\gamma}\bol{w}.\label{eq32}
\end{equation}
Evidently, the vector field \eqref{wxi} satisfies \eqref{eq32} with
\begin{equation}
P=0,~~~~\gamma=1-\frac{1}{f^2}.
\end{equation}

Finally, we remark that \eqref{wxi} 
is well defined along the toroidal axis $re^{-\epsilon Q}\rightarrow r_0$, $z\rightarrow 0$ provided that
$f$ exists in this limit (this is the case of $f^2=\exp\left\{ \Psi_{\epsilon}\right\}$ considered above).


\section{Considerations on magnetohydrodynamic equilibria, steady Euler flows, and quasisymmetry}

In this last section we discuss some aspects pertaining to the application of the theory developed in the previous sections to the analysis of equation 
\eqref{eq2}. 

First recall that solutions $\bol{w}$ of equation \eqref{eq2} are solutions of equation \eqref{eq1}. 
Therefore, the space of solutions of equation \eqref{eq2} is a subset of the space of solutions of equation \eqref{eq1}. 
Next, observe that the difference between equation \eqref{eq1} and equation \eqref{eq2} is that 
while in the former the vector field $\lr{\nabla\cp\bol{w}}\cp\bol{w}$ is only required to lie along $\nabla\Psi$, in the latter 
these two vector fields must coincide. 
Hence, in addition to the orthogonality between $\nabla\cp\bol{w}$ and $\nabla\Psi$ as described by equation \eqref{eq1X}, a further condition exists on the
magnitude of the component of $\lr{\nabla\cp\bol{w}}\cp\bol{w}$ along $\nabla\Psi$. 
In particular, enforcing the Clebsch representation
$\bol{w}=\nabla\Psi\cp\nabla\Theta$, equation \eqref{eq2} can be written as
\begin{equation}
\left[\nabla\cp\lr{\nabla\Psi\cp\nabla\Theta}\right]\cp\lr{\nabla\Psi\cp\nabla\Theta}=\left[\nabla\Theta\cdot\nabla\cp\lr{\nabla\Psi\cp\nabla\Theta}\right]\nabla\Psi
-\left[\nabla\Psi\cdot\nabla\cp\lr{\nabla\Psi\cp\nabla\Theta}\right]\nabla\Theta=\nabla\Psi.
\end{equation}
Hence, one obtains the system of equations
\begin{subequations}
\begin{align}
&\nabla\cdot\left[\nabla\Theta\cp\lr{\nabla\Psi\cp\nabla\Theta}\right]=-1,~~~~\nabla\cdot\left[\nabla\Psi\cp\lr{\nabla\Theta\cp\nabla\Psi}\right]=0~~~~{\rm in}~~\Omega,\label{redeq2a}\\&\Psi={\rm constant}~~~~{\rm on}~~\p\Omega.\label{redeq2b}
\end{align}\label{redeq2}
\end{subequations}
While in the study of equation \eqref{redeq0} the function $\Psi$ was given, it is convenient 
to regard system \eqref{redeq2} as coupled
partial differential equations for the
unknowns $\Psi$ and $\Theta$. 
Indeed, one expects that fixing $\Psi$ will prevent, in general, the existence of regular solutions 
$\Theta$ fulfilling both equations in \eqref{redeq2a}. 
Notice that boundary conditions \eqref{redeq2b} on $\Psi$ have been imposed to ensure that $\bol{w}\cdot\bol{n}=0$ on $\p\Omega$.  
Evidently, a solution $\lr{\Psi,\Theta}$ of system \eqref{redeq2} provides a solution $\bol{w}=\nabla\Psi\cp\nabla\Theta$ of equation \eqref{eq2}. 

It is worth observing that if the condition $\nabla\cdot\bol{w}=0$ is dropped in equation \eqref{eq2}, it is possible to find explicit solutions of $\lr{\nabla\cp\bol{w}}\cp\bol{w}=\nabla\Psi$ 
that break axial symmetry. Considering axially symmetric toroidal surfaces corresponding to level sets of $\Psi=\frac{1}{2}\left[\lr{r-r_0}^2+z^2\right]$, examples include vector fields of the type
\begin{equation}
\bol{w}=\sqrt{C-2\Psi^2}\nabla\vartheta+g\lr{\varphi}\nabla\varphi,
\end{equation}
with $\vartheta=\arctan\lr{z/\lr{r-r_0}}$ the poloidal angle, 
$\varphi=\arctan\lr{y/x}$ the toroidal angle, $C>0$
a sufficiently large real constant, and $g$ any periodic function of $\varphi$.  

As in the case of equation \eqref{redeq0}, equations
\eqref{redeq2a} admits a variational formulation.  
The target energy functional is
\begin{equation}
E_{\Omega}'=\int_{\Omega}\lr{\frac{1}{2}\abs{\nabla\Psi\cp\nabla\Theta}^2-\Psi}\,dV.\label{Hprime}
\end{equation}
where the variable $\Psi=P$ plays the role of the mechanical pressure $P$ in the context of magnetohydrodynamics, and corresponds to the sum  $\Psi=-P-\frac{1}{2}\bol{v}^2$ in the hydrodynamic  interpretation with $\bol{v}$ the fluid velocity.  
Assuming $\delta\Psi=\delta\Theta=0$ on $\p\Omega$, we have
\begin{equation}
\delta E_{\Omega}'=-\int_{\Omega}\delta\Psi\left\{1+\nabla\cdot\left[\nabla\Theta\cp\lr{\nabla\Psi\cp\nabla\Theta}\right]\right\}\,dV-\int_{\Omega}\delta\Theta 
\nabla\cdot\left[\nabla\Psi\cp\lr{\nabla\Theta\cp\nabla\Psi}\right]\,dV.
\end{equation}
Hence, stationary points of the functional $E_{\Omega}'$ assign solutions of \eqref{redeq2a}. 
Now suppose that solutions $\lr{\Psi,\Theta}$ of \eqref{redeq2} are sought 
in the Sobolev space $H^1\lr{\Omega}$ with norm $\Norm{\cdot}_{H^1\lr{\Omega}}$. 
From \eqref{Hprime} it is clear that the functional $E'_{\Omega}$ is not coercive, i.e. it does not
satisfy a condition of the form 
$E_{\Omega}'\geq c_1 \Norm{\Psi}_{H^1\lr{\Omega}}^2+c_2
\Norm{\Theta}_{H^1\lr{\Omega}}^2+C$ for some constants with  $c_1,c_2,C\in\mathbb{R}$, $c_1>0$, and $c_2>0$. 
Indeed, the value of \eqref{Hprime} can be kept finite, $\abs{E'_{\Omega}}<\infty$, 
even if $\Norm{\Psi}_{H^1\lr{\Omega}},\Norm{\Theta}_{H^1\lr{\Omega}}\rightarrow\infty$ by setting $\Theta=\Psi$ while taking $\Norm{\nabla\Psi}_{L^2\lr{\Omega}}=\Norm{\nabla\Theta}_{L^2\lr{\Omega}}\rightarrow\infty$ where $\Norm{\cdot}_{L^2\lr{\Omega}}$ denotes the standard $L^2\lr{\Omega}$ norm.  
The lack of coercivity prevents
the application of variational methods \cite{Struwe} to establish the existence of a relative minimizer of $E'_{\Omega}$, and thus a solution of \eqref{redeq2} in the relevant function space. 
 
It is worth however explaining why the situation is different if the variable $\Psi$ is fixed, i.e. if one considers equation \eqref{eq1X} arising from the functional $E_{\Omega}$ of \eqref{E} in the context of equation \eqref{eq1}.  
This will also provide explicit proof of the (weak) solvability of equation \eqref{redeq0} for the unkwon $\rho$ on each toroidal surface $\Psi={\rm constant}$. 
Consider the setting of theorem 1 where $\Psi$ is smooth, 
perform the change of variables $\Theta=\mu+\rho$,  
and use curvilinear coordinates  $\lr{x^1,x^2,x^3}=\lr{\mu,\nu,\Psi}$ to express $E_{\Omega}$ as follows
\begin{equation}
E_{\Omega}=\frac{1}{2}\int_{U}d\Psi\int_{D}\lr{\sum_{i,j=1}^2a_{ij}\rho_{i}\rho_{j}+2\sum_{i=1}^2a_{\mu i}\rho_{i}+a_{\mu\mu}}\,J d\mu d\nu,
\end{equation}
where $a_{ij}$ are the components of the symmetric positive definite matrix $A$ encountered in equation \eqref{A}.
Assuming that $\rho$ 
is periodic in $D$, 
integration by parts gives
\begin{equation}
E_{\Omega}=\frac{1}{2}\int_{U}d\Psi\int_D\left[\sum_{i,j=1}^2Ja_{ij}\rho_i\rho_j-2\rho\sum_{i=1}^2\frac{\p \lr{Ja_{\mu i}}}{\p x^i}+Ja_{\mu\mu}\right]\, d\mu d\nu.
\end{equation}
For each $\Psi\in U$ we may therefore identify an energy functional 
\begin{equation}
E_D=\frac{1}{2}\int_D\left[\sum_{i,j=1}^2Ja_{ij}\rho_i\rho_j-2\rho\sum_{i=1}^2\frac{\p \lr{Ja_{\mu i}}}{\p x^i}+Ja_{\mu\mu}\right]\, d\mu d\nu\geq
\lambda \Norm{\nabla_{\lr{\mu,\nu}}\rho}_{L^2\lr{D}}^2-2c\Norm{\rho}_{L^2\lr{D}}-C.
\end{equation}
Here, $\lambda$, $c$ and $C$ are positive real constants, 
and we used the strict ellipticity of $Ja_{ij}$. 
Recalling that  $\langle\rho\rangle=0$ and 
applying the Poincar\'e 
inequality, 
we further obtain
\begin{equation}
E_D\geq 
\frac{\lambda}{4}\Norm{\rho}^2_{H^1\lr{D}}+\lr{\frac{\lambda}{4}\Norm{\rho}_{H^1\lr{D}}-2c}\Norm{\rho}_{H^1\lr{D}}-C\geq \frac{\lambda}{4}\Norm{\rho}_{H^1\lr{D}}^2-4\frac{c^2}{\lambda}-C.
\end{equation}
This shows that $E_D$ is a coercive functional with respect to the $H^1\lr{D}$ norm since $E_D\rightarrow\infty$ when $\Norm{\rho}_{H^1\lr{D}}\rightarrow\infty$.  
Since $E_D$ is also sequentially lower-semicontinuous, for each $\Psi$ there exist a relative minimizer $\rho\in H^1_{\rm per}\lr{D}$ of the functional $E_D$,
which corresponds to a solution of equation \eqref{redeq0}.

We conclude this section with an observation concerning the existence of quasisymmetric solutions of equation \eqref{eq2}, i.e. solutions of equation \eqref{eq2} that satisfy the additional property
\begin{equation}
\bol{u}\cp\bol{w}=\nabla g\lr{\Psi},~~~~\bol{u}\cdot\nabla \bol{w}^2=0,~~~~\nabla\cdot\bol{u}=0~~~~{\rm in}~~\Omega,\label{QS}
\end{equation}
for some function $g\lr{\Psi}$ such that $\nabla g\neq\bol{0}$ and some vector field $\bol{u}$ called the quasisymmetry of $\bol{w}$. 
The property \eqref{QS} is a desirable feature
for the confining magnetic field in nuclear fusion reactors known as stellarators, 
because it ensures steady confinement of the burning plasma within a finite volume of space \cite{Rod}. 
In this regard, we have:
\begin{proposition}
Suppose that $\bol{\xi}\in L^2_H\lr{\Omega}$ is a harmonic vector field in a toroidal domain $\Omega$ foliated by nested toroidal surfaces corresponding to contours of a function $\Psi\in C^1(\bar{\Omega})$. Further assume that 
\begin{equation}
\bol{\xi}\cdot\nabla\Psi=0~~~~{\rm in}~~\Omega,
\end{equation}
and that $\abs{\bol{\xi}}^2=\abs{\bol{\xi}}^2\lr{\Psi}$. 
Then, the vector field $\bol{w}=f\lr{\Psi}\bol{\xi}\in H^1_{\sigma\sigma}\lr{\Omega}$, with $f\in C^1(\bar{\Omega})$,  
solves \eqref{eq2} and is quasisymmetric with quasisymmetry
\begin{equation}
\bol{u}=\bol{\xi}\cp\nabla\Psi\in L^2_{\sigma}\lr{\Omega}.
\end{equation}
\end{proposition}
The proof of the above statement can be obtained by evaluating equations \eqref{eq2} and \eqref{QS}. 
We remark that, however, the
requirement $\abs{\bol{\xi}}^2=\abs{\bol{\xi}}^2\lr{\Psi}$ that the modulus of the harmonic vector field $\bol{\xi}$ is a function of $\Psi$ is a stringent condition related to the notion of isodynamic magnetic field \cite{Bern}. Therefore, the existence of such configurations is nontrivial.

\section{Concluding remarks}

In this paper, we studied equation \eqref{eq1}, 
which determines a solenoidal vector field 
with the property that both the vector field ans its curl are foliated by a family of nested toroidal surfaces. 
Equation \eqref{eq1} represents a generalization of 
an equation encountered in magnetohydrodynamics 
and fluid mechanics (equation \eqref{eq2}) describing 
equilibrium magnetic fields 
and steady Euler flows. 
At present, a general theory concerning the existence of solutions of equation \eqref{eq2} is not available due to the mathematical difficulty originating from its nontrivial characteristic surfaces.  
Analysis of the simpler problem posed by equation \eqref{eq1} may therefore provide useful insight into the nature of the space of solutions of equation \eqref{eq2}. 

In theorem 1 we showed that nontrivial solutions in the class $C^{\infty}\lr{\Omega}$ of equation \eqref{eq1}, where $\Omega$ is a hollow toroidal volume, always exist for a given family of smooth nested toroidal surfaces.
The proof relies on the reduction of equation \eqref{eq1} to a two-dimensional linear elliptic second order partial differential equation \eqref{redeq0} for each toroidal surface with the aid of Clebsch parameters. Regular periodic solutions for these equations exist by elliptic theory, and  
can be used to determine the desired smooth solution of problem \eqref{eq1}. 
In section 5, an example of numerical solution was also computed, while in section 6 examples of smooth solutions  in toroidal volumes were  constructed analytically such that both the bounding surface and the solution are not invariant under continuous Euclidean isometries. 
Such solutions can be regarded as solutions of anisotropic magnetohydrodynamics \eqref{eq3}. 

The results obtained above concerning equation \eqref{eq1} entail a number of consequences for the 
problem described by equation \eqref{eq2}. 
First, the formulation of equation \eqref{eq2} in terms of Clebsch potentials (equation \eqref{redeq2}) discussed in section 7 suggests that 
simultaneous optimization of the Clebsch potentials $\Psi$ and $\Theta$ is needed to find solutions. That is, the shape of the toroidal surfaces $\Psi$ (and possibly the profile of the domain $\Omega$ itself) should be adjusted together with the variable $\Theta$ 
to accommodate the solution within $\Omega$. 
Secondly, if solutions are sought in the form $\bol{w}=f\lr{\Psi}\bol{\xi}$ 
of \eqref{wxi}, where $\bol{\xi}$ is a harmonic vector field in $\Omega$, solving \eqref{eq2} amounts to
finding a harmonic vector field that is foliated 
by toroidal surfaces $\Psi$ and such that the modulus
$\abs{\bol{\xi}}^2$ is itself a function of $\Psi$. 
Finally, as observed in proposition 2 of section 7, 
if such kind of solution could be found, it would also guarantee quasisymmetry, and thus magnetic confinement of a plasma within a finite volume of space as desirable in nuclear fusion applications. 

\section*{Acknowledgment}
N.S. is grateful to T. Yokoyama for useful discussion. 

\section*{Statements and declarations}

\subsection*{Data availability}
Data sharing not applicable to this article as no datasets were generated or analysed during the current study.

\subsection*{Funding}
The research of NS was partially supported by JSPS KAKENHI Grant No. 21K13851.

\subsection*{Competing interests} 
The authors have no competing interests to declare that are relevant to the content of this article.



\end{document}